\documentclass[12pt,leqno]{amsart}
\usepackage{latexsym,amssymb,amsmath,multicol, rotating,lscape}
\usepackage{bbm}
\usepackage{letltxmacro}
\usepackage{tikz}
\usetikzlibrary{calc,decorations.markings}

\usepackage{hyperref}

 2

\theoremstyle{plain}
\newtheorem{thm}{Theorem}[section]
\newtheorem{lem}[thm]{Lemma}

\theoremstyle{definition}
\newtheorem{defn}[thm]{Definition}
\newtheorem{exmp}[thm]{Example}

\newcommand{\thmref}[1]{Theorem~\ref{#1}}
\newcommand{\lemref}[1]{Lemma~\ref{#1}}

\theoremstyle{remark}

\newcommand*{\Z}{\mathbb{Z}}

\renewcommand*{\a}{\alpha}
\renewcommand*{\b}{\beta}

\renewcommand*{\l}{\lambda}

\renewcommand*{\t}{\tau}

\newcommand*{\z}{\zeta}

\newcommand*{\De}{\Delta}
\newcommand*{\Ga}{\Gamma}

\newcommand*{\im}{\operatorname{Im}}

\newcommand*{\err}{\operatorname{erf}}

\newcommand*{\vt}[1]{\left\lvert #1 \right\rvert}

\newcommand*{\pangle}[1]{\left\langle #1\right\rangle}

\newcommand*{\lp}{\left(}
\newcommand*{\rp}{\right)}

\newcommand*{\bsm}{\left(\begin{smallmatrix}}
	\newcommand*{\esm}{\end{smallmatrix}\right)}

\newcommand*{\bpm}{\begin{pmatrix}}	
	\newcommand*{\epm}{\end{pmatrix}}

\newcommand*{\abcd}{\begin{pmatrix}a&b\\c&d\end{pmatrix}}

\begin{document}
	
	\title[Construction of Jacobi forms using adjoint Jacobi-Serre derivative]
	{Construction of Jacobi forms using adjoint of  Jacobi-Serre derivative}

	\author{Mrityunjoy Charan and Lalit Vaishya }
	\address{(Mrityunjoy Charan) The Institute of Mathematical Sciences,
		CIT Campus, Taramani,
		Chennai - 600 113,
		Tamil Nadu, 
		India.}
	\email{mcharan@imsc.res.in}
	\address{(Lalit vaishya) The Institute of Mathematical Sciences,
		CIT Campus, Taramani,
		Chennai - 600 113,
		Tamil Nadu,
		India.}
	\email{lalitvaishya@gmail.com, lalitv@imsc.res.in}

	
	\date{\today}
	
	\maketitle

	\begin{abstract}
	In the article, we study the Oberdieck derivative defined on the space of weak Jacobi forms. We prove that the Oberdieck derivative maps a Jacobi form to a Jacobi form. Moreover, we study the adjoint of Oberdieck derivative of a Jacobi cusp form with respect to Petersson scalar product defined on the space of Jacobi forms. As a consequence, we also obtain the adjoint of Jacobi-Serre derivative (defined in an unpublished work of Oberdieck).  As an application, we obtain certain relations among the Fourier coefficients of Jacobi forms.

	%We compute adjoint of Jacobi-Serre derivaive of Jacobi forms with respect to Petersson scalar product. 			
	\end{abstract}
	
	\section{Introduction}
Modular forms are prominent  objects in number theory and it has numerous applications in other branches of Mathematics as well as in Physics. Construction of a  modular form is one of the fascinating  conundrum in the theory of modular forms. It is well-known that the derivative of  a classical modular form is not necessarily a classical modular form, but a certain linear combination of derivative allows one to construct classical modular form. Rankin-Cohen brackets and Serre derivative are one of those approaches.  

Using the existence of the  adjoint of  a linear map and properties of Poincar\'e series, Kohnen \cite{kohnen} first constructed certain classical modular forms with the property that its Fourier coefficients  involve special values of certain Dirichlet series. In particular,  Kohnen constructed the adjoint map  of the product map by a fixed holomorphic cusp form with respect to the usual Petersson scalar product.  Later on, following the idea of Kohnen \cite{kohnen}, the construction of holomorphic modular form has been carried out using Rankin-Cohen brackets  by Herrero \cite{herrero},  and Serre derivative by Kumar \cite{Arvind2017}. %Kumar \cite{Arvind2017} studied the adjoint of Serre derivative and using the theory of nearly holmorphic modular forms he constructed cusp forms which also involves special values of Dirichlet series.

% Likewise, construction of automorphic form, such as Hilbert modular form, Jacobi form and  Siegel modular form etc,  is one of the  interesting quandary  in the theory of automorphic form. 
Jacobi forms are a natural generalization of classical (holomorphic) modular forms to several variable case. Jacobi forms were first studied systematically by Eichler and Zagier in \cite{EZ1985}. Jacobi forms are  key objects in building a bridge between classical modular forms and Siegel modular forms. This insight leads to prove a famous Saito-Kurokawa conjecture. In the case of Jacobi forms,  Choie et al \cite{ckk} and Sakata \cite{sakata} constructed Jacobi forms using the idea of  Kohnen \cite{kohnen}.
There are many work in this direction which allows to construct Jacobi forms. For details, we refer \cite{{choie1}, {a-s}, {r-s}, {zag}}.

 In 2014, Oberdieck \cite{OB2014} constructed the  Jacobi-Serre derivative on the space of Jacobi forms, which is an obvious  generalization of Serre derivative  on the space of modular forms. Analogous to Ramanujan derivative equation associated to classical Eisenstien series, he gave Ramanujan equation for Jacobi forms of  index $1$  using   Jacobi-Serre derivative. Following the work of Oberdieck \cite{OB2014}, Choie et al \cite{CDMR2021} constructed a derivative on the space of weak Jacobi forms and proved that it maps a weak Jacobi form to a weak Jacobi form.

 In this article, we analyse Oberdeick derivative and prove that it maps a Jacobi form to a Jacobi form. Moreover, we also observe that the Jacobi-Serre derivative can be expressed in terms of the Heat operator and the  Oberdieck derivative. Furthermore, we obtain the adjoint of Oberdeick derivative map.
 As a consequence, we compute the adjoint of Jacobi-Serre derivative map.   As an application, we also obtain certain relations among the Fourier coefficients of Jacobi forms.

The structure of article is as follows. In Section $2$, we briefly recall the notion of Jacobi forms and define various derivatives acting on the space of Jacobi forms and state our main result. In section $3$, we talk about adjoint of these operators and state the important results. In section $4$, We prove certain lemma which are key ingredient to prove our result. Finally, in Section $5$, we prove our results and give certain application.

 \section{Preliminaries}
% We recall Jacobi form and its certain property.
 The Jacobi group $\Gamma^J:=SL_2(\mathbb{Z}) \ltimes (\mathbb{Z} \times \mathbb{Z})$ acts on $\mathbb{H} \times \mathbb C$ by the following action: 
$$
\left( \left(\begin{array}{cc}
a & b\\
c & d\\
\end{array}\right), (\lambda, \mu) \right)\cdot(\tau, z)=
\left( \frac {a\tau+b}{c\tau+d}, \frac{z+\lambda \tau+ \mu}{c\tau+d} \right).
$$
Let $k, m$ be fixed positive integers. For a complex-valued function $\phi$ on $ \mathbb{H} \times \mathbb C$ and \linebreak $\gamma= \left( \left(\begin{array}{cc}
	a & b\\
	c & d\\
	\end{array}\right), (\lambda, \mu) \right) \in \Gamma^J,$ we define the stoke operator $|_{k, m}$ given by  
	
	$$
	\left(\phi|_{k, m} \gamma\right)(\tau,z):= (c\tau+d)^{-k} e^{2 \pi i m\small(-\frac{c(z+\lambda \tau+\mu)^2)}{c\tau+d}+ \lambda^2\tau+2\lambda z)} \phi (\gamma\cdot (\tau, z)).
	$$
	
	\begin{defn}
		A Jacobi form of weight $k$ and index $m$ on $\Gamma^J$  is a holomorphic function $\phi: \mathbb{H} \times \mathbb C \rightarrow \mathbb C$ satisfying the followings:
		$$\phi|_{k, m} \gamma=\phi, ~ \forall \gamma \in \Gamma^J,$$ 
		and having a Fourier expansion of the form 
		$$
		\phi(\tau, z)= \sum_{\substack{n \ge 1, r \in \mathbb{Z},\\ 4nm-r^{2}\ge 0}} c(n, r) q^n \zeta^r, 
		$$
		where $q=e^{2 \pi i \tau}$ and $\zeta=e^{2 \pi i z}$.
		Further, we say that   $\phi$ is a cusp form if $$c(n, r) \neq 0 \implies 4nm-r^{2}>0.$$ 
	\end{defn}
	\noindent We denote the space of all Jacobi forms of weight $k$ and index $m$ and the subspace of all Jacobi cusp forms of weight $k$ and index $m$ by $J_{k, m}$ and $J_{k, m}^{cusp}$, respectively. The Petersson scalar product $\langle,  \rangle$ on the space  $J_{k, m}^{cusp}$ is given as follows.  For  $\phi,~\psi\in J_{k, m}^{cusp}$, 
	$$
	\langle \phi, \psi \rangle=\int\limits_{\Gamma^J\setminus \mathbb{H}\times \mathbb{C}} \phi(\tau,z)\overline{\psi(\tau, z)}  v^k e^{\frac{-4\pi m y^2}{v}}dV_J,
	$$
	where $dV_J=\dfrac{dudvdxdy}{v^3}$ is an invariant measure under the action of $\Gamma^J$ on $\mathbb{H} \times \mathbb C$ with $\tau=u +i v, z=x+iy$ and 
	${\Gamma^J}\setminus \mathbb{H}\times \mathbb{C}$ is a fundamental domain for the action of $\Gamma^J$ on $\mathbb{H} \times \mathbb C$.
	The space $(J_{k, m}^{cusp}, \langle, \rangle)$ is a finite dimensional Hilbert space. 
	
	\begin{exmp}
		Let $k\geq 4$ be an even integer. The Jacobi Eisenstein series of weight $k$ and index $m$ is defined as 
		\begin{equation*}
		E_{k,m}(\tau, z)=\frac{1}{2}\sum_{\substack{c,d \in \mathbb{Z}\\(c,d)=1}}\sum_{\lambda \in \mathbb{Z}}(c\tau +d)^{-k}e^m \left(\lambda^2\frac{a\tau +b}{c\tau +d}+2\lambda \frac{z}{c\tau +d}-\frac{cz^2}{c\tau +d}\right),
		\end{equation*}
		where $a$ and $b$ are integers such that $\left(\begin{array}{cc}
		a & b\\
		c & d
		\end{array}\right)\in SL_2(\mathbb{Z})$. Then $E_{k,m}$  is a Jacobi form of weight $k$ and index $m$ for the Jacobi group  $\Gamma^J.$
	\end{exmp}
	
	\begin{exmp}
		Let $m, n$ and $r$ be fixed integers with $4mn-r^{2}>0$. The $(n, r)$-th Jacobi-Poincar{\'e} series of weight $k$ and index $m$ is defined as
		\begin{equation}\label{jacobi-poincare}
		P_{k, m; (n, r)} (\tau, z):=\sum_{\gamma \in \Gamma^J_{\infty}\setminus \Gamma^J} e^{2 \pi i (n \tau+ rz)}|_{k, m} \gamma .
		\end{equation}
		Here $\Gamma^J_{\infty}:=\left\{ \left( \left(\begin{array}{cc}
		1 & t\\
		0 & 1\\
		\end{array}\right), (0, \mu) \right) :~ t, \mu \in \mathbb{Z}  \right\}.$ 
		It is well-known that the $(n, r)$-th Jacobi-Poincar{\'e} series $P_{k, m; (n, r)} \in J_{k, m}^{cusp}$ for $k >2$. % (see \cite{g-k-z}). 
	\end{exmp}
	\noindent This Poincar{\'e} series $P_{k, m; (n, r)}$ has the following property.
	
	\begin{lem}\label{jacobi-poincare-lemma}
		Let $\phi \in J_{k, m}^{cusp}$ with Fourier expansion 
		$$\phi(\tau, z)=\sum_{\substack{n \ge 1, r \in \mathbb{ Z},\\ 4nm -r^{2}>0}} c(n, r) q^n \zeta^r.$$ Then 
		\begin{equation*}
		\langle \phi,  P_{k, m; (n, r)} \rangle= \alpha_{k, m} (4mn-r^2)^{-k+\frac{3}{2}} c(n, r),
		\end{equation*}
		where $$\alpha_{k, m}=\dfrac{m^{k-2} \Gamma(k-\frac{3}{2})} {2\pi^{k-\frac{3}{2}}}.$$
	\end{lem}
	\noindent The following lemma tells about the growth of the Fourier coefficients of a Jacobi form. 
	
	\begin{lem}[Choie, Kohnen \cite{choie5}]\label{Jacobi-convergence}
		If $\phi\in J_{k, m}$ and $k>3$ with Fourier coefficients $c(n, r),$ then 
		$$
		c(n, r) \ll |r^2-4nm|^{k-\frac{3}{2}}. 
		$$  
		Moreover, if $\phi$ is a Jacobi cusp form, then
		$$
		c(n, r) \ll |r^2-4nm|^{\frac{k}{2}-\frac{1}{2}}.
		$$  
	\end{lem}
	\noindent
	For more details on the theory of Jacobi forms, we refer to \cite{EZ1985}.

	\subsection{Differentials on Jacobi form} 
	Let $D_{\tau}:= \frac{1}{2 \pi i} \frac{\partial}{\partial \tau}$ and $D_{z}:= \frac{1}{2 \pi i} \frac{\partial}{\partial z}$ denote the partial derivative operators. For a matrix $\gamma = \left(\begin{array}{cc}
	a & b\\
	c & d\\
	\end{array}\right) \in SL_{2}(\mathbb{Z})$, let us write $\chi(\gamma) = \frac{c}{c \tau +d}$. Then, for any $F \in \mathbb{C}[[q, \zeta, \zeta^{-1}]]$, we have the following transformation properties for  $D_{\tau} F$ and $D_{z}F$ given by
	\begin{equation*}\label{DTAFT}
	D_{\tau} (F|_{k,m} \gamma) = \left( m \chi(\gamma)^{2} - \frac{k}{2 \pi i} \chi(\gamma)  \right)( F|_{k,m} \gamma) - \chi(\gamma) (D_{z} F)_{|_{k+1,m}} \gamma + D_{\tau} (F)|_{k+2,m} \gamma,
	\end{equation*}
	and
	\begin{equation*}\label{DZFT}
	D_{z} (F|_{k,m} \gamma) = -2m ( F|_{k,m} \gamma) + (D_{z}F)_{|_{k+1,m}} \gamma.
	\end{equation*}
	From the above transformation properties, the derivatives  $D_{\tau}F$ and $D_{z} F$ are no longer Jacobi forms. Still, one can expect that a certain combination of these derivatives leads to construct a Jacobi form. Below, we produce the details of differential operators that lead to construct a Jacobi form.

	\smallskip
	\noindent
	\textbf{Heat operator on Jacobi form:} 
	For an integer $m,$ Choie \cite{choie11} has defined  the heat operator $\mathcal{L}_{m}$ given by  
	$$
	\mathcal{L}_{m}:=\frac{1}{(2\pi i)^2}\left(8 \pi im \frac{\partial}{\partial \tau}-\frac{\partial^2}{\partial z^2}\right) = 4mD_{\tau}-D_{z}^{2}. 
	$$
	The operator $\mathcal{L}_{m}$ is natural in the context of Jacobi form because it acts on monomials $q^{\ell}\zeta^{w}$ by multiplication by $(4 \ell m-w^{2} )$, i.e., $\mathcal{L}_{m} (q^{\ell}\zeta^{w}) = (4 \ell m-w^{2} ) q^{\ell}\zeta^{w} $, and hence it preserves the actions of lattices.  
	It is known that this operator satisfies the following functional equations; for any  $\gamma = \left(\begin{array}{cc}
	a & b\\
	c & d\\
	\end{array}\right) \in SL_{2}(\mathbb{Z})$ and any $X \in \mathbb{Z}^{2}$,
	\begin{equation}\label{Heat-trans}
	\begin{split}
	\mathcal{L}_{m}(\phi)|_{k+2, m} \gamma &= \mathcal{L}_{m}(\phi|_{k, m} \gamma) + \frac{2m(2k-1)}{2 \pi i} \frac{c}{c \tau+d} (\phi|_{k, m} \gamma),  \\
	\mathcal{L}_{m}(\phi)|_{m} X &= \mathcal{L}_{m}(\phi|_{ m} X).
	\end{split}
	\end{equation}
	From \eqref{Heat-trans}, it is clear that the heat operator does not map a Jacobi form to a Jacobi form. But, if we consider the operator $\mathcal{L}_{k,m} := \mathcal{L}_{m}- \frac{m(2k-1)}{6}E_{2}$, then it is easy to see that it  maps Jacobi form of weight $k$ index $m$ to  a Jacobi form of weight $k+2$ index $m$, i.e.,
	\begin{equation*}
	\begin{split}
	\left( \mathcal{L}_{m}- \frac{m(2k-1)}{6}E_{2} \right)\phi|_{k+2, m} \gamma &= \left( \mathcal{L}_{m}- \frac{m(2k-1)}{6}E_{2} \right) \phi \;\;\; {\rm for ~ all ~}  \phi \in J_{k, m}.
	\end{split}
	\end{equation*}
	More precisely,  we obtain a  linear differential operator
	\begin{equation*}\label{LDO1}
	\begin{split}
	\mathcal{L}_{k,m}: J_{k,m}^{cusp} \longrightarrow J_{k+2,m}^{cusp} \;  \quad \text{given by} \quad  \phi \mapsto \left( \mathcal{L}_{m}- \frac{m(2k-1)}{6}E_{2} \right)\phi.
	\end{split}
	\end{equation*}
	
	\noindent
	\textbf{Deformed Eisenstein series and Oberdieck derivation:}
	Let $B_{n}$ denote the $n$-th Bernoulli numbers defined as; $\frac{t}{e^{t}-1} = \displaystyle{\sum_{n \ge 1}} B_{n} \frac{t^{n}}{n!}$. In particular, $B_{0}=1, B_{1}= -\frac{1}{2}, B_{2} = \frac{1}{6}, \cdots$. Oberdieck \cite{OB2014} defined the deformed or twisted Eisenstein series $J_{n}(\tau, z)$ (for $n \ge 0$) given by
	\begin{equation}\label{TwistESeries}
	\begin{split}
	J_{n}(\tau, z) = \delta_{n,1}\frac{\z}{\z-1} + B_{n}- n \sum_{k,r \ge 1} r^{n-1} (\zeta^{k}+ (-1)^{n} \zeta^{-k}) q^{kr}.
	\end{split}
	\end{equation}
	The name of $J_{n}$ reminds us of the fact that they restricted to the classical Eisenstein series $E_{2k}$ at $z=0$, i.e., 
	$$
	J_{2k}(\tau,0)= B_{2k} E_{2k} \quad {\rm and} \quad J_{2k+1}(\tau,0) \equiv 0.
	$$ 
 The deformed Eisenstein series $J_{1}(\tau, z)$ arises as a logarithmic derivative (with respect to $z$) of Jacobi theta function $\theta_{1}(\tau,z)$ and $J_{n}(\tau, z)$ can be given using $\theta_{1}(\tau,z)$ and its derivatives (with respect to $z$). The deformed Eisenstein series $J_{n}(\tau, z)$ transforms like a Jacobi form of weight $n$ and index $0$ with certain additional lower terms.  The function $J_{1}(\tau, z)$ and $J_{2}(\tau, z)$ satisfies the following transformation properties,  respectively.
	\begin{equation}\label{J1Transf}
	\begin{split}
	\forall (\lambda, \mu) \in \mathbb{Z}, \quad  &  {J_{1}}|_{0 }(\lambda,\mu) = J_{1} -\lambda, \\
	\forall \gamma \in SL_{2}(\mathbb{Z}), \quad & {J_{1}}|_{1,0 } \gamma  = J_{1} + \chi(\gamma),  
	\end{split}
	\end{equation}
	and 
	%The function  satisfies the following transformation properties:
	\begin{equation}\label{J2Transf}
	\begin{split}
	\forall (\lambda, \mu) \in \mathbb{Z}, \quad  &  {J_{2}}|_{0 }(\lambda,\mu) = J_{2} - 2 \lambda J_{1} +{\lambda}^{2}, \\
	\forall \gamma \in SL_{2}(\mathbb{Z}), \quad & {J_{2}}|_{2,0 } \gamma  = J_{2} + 2 \chi(\gamma) J_{1} + \chi(\gamma)^{2},  
	\end{split}
	\end{equation}
	where $\chi(\gamma)$ is a function from $\mathbb{H} \times \mathbb{C}$ to $\mathbb{C}$ given by 
	$$
	\chi(\gamma)(\tau, z) = \frac{c z}{ c\tau +d}.
	$$
	For details, we refers to \cite[Section 1]{OB2014}  and  \cite[Section 2]{CDMR2021}.

A weak Jacobi form of weight $k$ and index $m$ is a holomorphic function $\Phi : \mathbb{H} \times \mathbb{C} \longrightarrow \mathbb{C}$  invariant
by the action of  the stoke operator $|_{k, m}$  of the Jacobi group $\Gamma^{J}$, and whose Fourier expansion is  given by
$$
\Phi(\tau, z)= \sum_{\substack{n \ge 1, r \in \mathbb{Z},\\ r^{2} \le 4nm+ m^{2}}} c(n, r) q^n \zeta^r.
$$
Let $\tilde{J}_{k,m}^{cusp}$ denote the $\mathbb{C}$-space of weak Jacobi forms. The  space $\tilde{J}_{k,m}^{cusp}$ of such functions is finite dimensional \cite{EZ1985}. Moreover, The space of  Jacobi forms $J_{k,m}^{cusp}$ is a proper subspace of the space of weak Jacobi forms $\tilde{J}_{k,m}^{cusp}$. Now, we define the notion of Oberdieck derivative defined on the space of weak Jacobi form $\tilde{J}_{k,m}^{cusp}$.
	
	\noindent
	\textbf{Oberdieck derivative:}  For a given $\phi \in \tilde{J}_{k,m}^{cusp} $, using the deformed Eisenstein series,  Choie et al \cite{CDMR2021} defined the notion of  Oberdieck derivative $\mathcal{O} \phi$  associated to  $\phi$ given by
	\begin{equation}\label{OBDerivative}
	\mathcal{O}\phi = D_{\tau} \phi - \frac{k}{12} E_{2} \phi - J_{1} D_{z} \phi + m J_{2} \phi.
	\end{equation}
	From  \cite[Section 4.3]{CDMR2021}, it is well-known that $\mathcal{O}$ maps $\tilde{J}_{k,m}^{cusp}$ to $\tilde{J}_{k+2,m}^{cusp}$ given by 
	$$\phi \mapsto  \left( D_{\tau}  - \frac{k}{12} E_{2}  - J_{1} D_{z}  + m J_{2} \right)  \phi. $$

	\noindent
	\textbf{Serre derivative on Jacobi form:}
	Let $f \in M_{k}(SL_{2}(\mathbb{Z}))$ be a classical modular form. The derivative $D_{\tau}f$ fails to be a modular form. To make it a modular form, one needs to add a multiple of a quasi-modular form (more precisely, Eisenstein series $E_{2}$). In particular, one can get a linear differential operator (namely, Serre derivative) which maps a modular form of weight $k$ to a modular form of weight $k+2$. The Serre derivative is defined as follows:
	$$
	{\partial}^{S} :M_{k}(SL_{2}(\mathbb{Z}))  \rightarrow M_{k}(SL_{2}(\mathbb{Z})), \qquad   {\partial}^{S}f := D_{\tau}f - \frac{k}{12} E_{2} f.  
	$$
	Finite dimensionality of the vector space $M_{k}(SL_{2}(\mathbb{Z}))$ allows one to get the following  Ramanujan equations for the Eisenstein series $E_{2}$, $E_{4}$ and $E_{6}$:
	\begin{equation}\label{Ram-Eq}
	{\partial}^{S}E_{2} + \frac{1}{12} E_{2}^{2} = \frac{1}{12} E_{4}, \quad  {\partial}^{S}E_{4}  = - \frac{1}{3} E_{6}, \quad {\rm and } \quad {\partial}^{S}E_{6} = - \frac{1}{2} E_{4}^{2}.
	\end{equation}

	Using the same strategy, one can define an analogue of the Serre derivative (for classical modular form) on the space of the Jacobi form of even weight and any index.  Oberdieck  \cite{OB2014} defined the Jacobi-Serre derivative ${\partial}^{J}$ on the space of Jacobi form of even weight and any index and it is given by  
	\begin{equation}\label{Jacobi-Serre-D}
	{\partial}^{J} (\phi)   = D_{\tau} \phi - \frac{k}{12} E_{2}~ \phi + \frac{1}{1-4m} \left(  D_{z}^{2} \phi -J_{1}  D_{z} \phi + m J_{2} ~ \phi - \frac{m}{6} E_{2}~ \phi   \right).
	\end{equation}
	Moreover, he proved the following result. 
	\begin{thm}\label{SerreDThm}\cite[Theorem 1]{OB2014}
		For all non-negative integers $k,m$, there is a differential operator %${\partial}^{J}$ given by 
		$$
		{\partial}^{J} :J_{2k,m} \longrightarrow J_{2k+2,m}
		$$
		such that for every $\phi(\tau,z) \in J_{2k,m}$, we have 
		$$
		({\partial}^{J}\phi) (\tau,0) = ({\partial}^{S}) \phi(\tau,0). 
		$$
	\end{thm} 
Analogue to the Ramanujan differential equations \eqref{Ram-Eq} associated to the classical Eisenstein series, Oberdieck established the differential equation associated to index $1$ Eisenstein series $E_{2,1}$, $E_{4,1}$ and  $E_{6,1}$   in \cite[Corollary 3]{OB2014}.      	

Let us consider the following weak Jacobi form
\begin{equation*}
	\phi_{-2,1}(\t, z)
	=\frac{\phi_{10,1}}{\De(\t)}.
\end{equation*}
 It is  one of the generators of the algebra of even-weight weak Jacobi forms \cite[Theorem 9.3]{1-2-3}. Let $\wp(\t, z)$ denote the  Weierstrasse $\wp$-function given by 
 \begin{equation*}
 	\wp(\t, z)
 	=\frac{1}{(2\pi i)^2}
 	\lp
 	\frac{1}{z^2}
 	+\sum_{n\ge 1}^{}
 	2(2n+1)
 	\zeta(2n+2)
 	E_{2n+2} z^{2n}
 	\rp
 \end{equation*}
 %be the Weierstrasse $\wp$-function. 
 Oberdieck  \cite{OB2014} defined the analogue of $E_2$ for Jacobi forms of
 index $1$ as follows:
 $$E_{2,1}(\t, z)
 := \phi_{-2,1}(\t, z)
 \lp
 E_2(\t) \wp(\t,z)
 -\frac{1}{12}
 E_4
 \rp.
 $$
Although $E _{2,1}$ has several properties similar to  $E _2 (\t )$ \cite[Lemma 13]{OB2014}, the definition is rather ad-hoc. Oberdieck mentioned that  it would be interesting to find a more conceptual approach to define $E_{2,1}(\t, z)$. Moreover,  Oberdieck \cite[Corollary 3] {OB2014} stated the Ramanujan equation for  Jacobi form of  index $1$ given as follows: 
	\begin{equation}\label{eq:ramanujan eq for jacobi}
		\begin{split}
	          \partial^{J}E_{2,1}+\frac{1}{12}E_2E_{2,1}+\frac{1}{16}E_4^{'}\phi_{-2,1}
	          &=-\frac{1}{12}E_{4,1},\\
	          \partial^{J}E_{4,1}
	          &=-\frac{1}{3}E_{6,1},\\
	          \partial^{J}E_{6,1}
	          &=-\frac{1}{2}E_4E_{4,1}
		\end{split}
	\end{equation}
	where $E _{4,1}$ and $E _{6,1}$ be the Jacobi-Eisenstein series of index $1$ and weight $4$ and $6$, respectively.
These differential equations becomes  the original Ramanujan equations obtained in  \eqref{Ram-Eq} (associated to classical Eisenstein series) by substituting  $z=0$ in \eqref{eq:ramanujan eq for jacobi}.

Let us recall the operator  $\mathcal{L}_{k,m}$  given by $\mathcal{L}_{k,m}=\mathcal{L}_{m}- \frac{m(2k-1)}{6}E_{2}$.	It is easy to observe that the  Jacobi-Serre derivative ${\partial}^{J}$ can be re-written as a linear combination of  $ \mathcal{L}_{k,m}$ and Oberdieck derivative $\mathcal{O}$ defined in \eqref{OBDerivative} . More precisely,
	\begin{equation}\label{Jacobi-Serre-DSumT}
	\begin{split}
	{\partial}^{J} \phi   %&= \frac{1}{1-4m} \left[ \left(-\mathcal{L}_{m}+ \frac{m(2k-1)}{6}E_{2}\right) \phi +  \left(   D_{\tau}  - \frac{k}{12} E_{2}  - J_{1} D_{z}  + m J_{2} \right) \phi \right]  \\
	& =  \frac{1}{4m-1} \left(\mathcal{L}_{k,m} \phi - \mathcal{O} \phi \right).
	\end{split}
	\end{equation}
	Moreover, one can see that 
	\begin{equation}\label{Obed-D}
	\begin{split}
	\mathcal{O} \phi=   \mathcal{L}_{k,m} \phi - (4m-1){\partial}^{J} \phi .
	\end{split}
	\end{equation}
	Choie et al \cite{CDMR2021}  proved that  the Oberdieck derivative of  weak Jacobi forms is again  a weak Jacobi form. From \eqref{Obed-D}, we conclude that if $\phi$ is a Jacobi form of weight $k$ and index $m$ then $\mathcal{O} \phi$ is not only a weak Jacobi form but a Jacobi form of weight $k+2$ and  index $m$. In particular, we have the following result. 
	\begin{thm}
		The Oberdieck derivative $\mathcal{O}$ maps $J_{k,m}^{cusp}$ to $J_{k+2,m}^{cusp}$ given by 
		$$
		\phi \mapsto  \left(   D_{\tau}  - \frac{k}{12} E_{2}  - J_{1} D_{z}  + m J_{2} \right) \phi .
		$$
	\end{thm}
%\newpage
	
	\section{Adjoint of certain linear maps on Jacobi form}
Kumar \cite{Arvind2017}  computed the adjoint of the Serre derivative ${\partial}^{S}$ with respect to the Petersson scalar product by using the properties of nearly holomorphic modular forms. The Fourier coefficients of a cusp form, constructed in this fashion, involve special values of certain shifted Dirichlet series. 

From \eqref{Jacobi-Serre-DSumT}, we see that Jacobi-Serre derivative ${\partial}^{J}$, acting on  the space of Jacobi cusp form, can be re-written as a linear combination of weight raising operator $ \mathcal{L}_{k,m}$ and Oberdieck derivative $\mathcal{O}$, i.e.,
\begin{equation*}\label{JSD}
\begin{split}
{\partial}^{J}  & =  \frac{1}{4m-1} \left(\mathcal{L}_{k,m} - \mathcal{O} \right).
\end{split}
\end{equation*}
Let ${\partial}^{J *}$ denote the adjoint of the Jacobi-Serre derivative ${\partial}^{J}$, i.e., 
$${\partial}^{J *}: {J}_{k+2,m}^{cusp} \longrightarrow {J}_{k+2,m}^{cusp}.$$
From \eqref{JSD} and properties of adjoint ${\partial}^{J *}$, we see that 
\begin{equation*}\label{JSAdj}
\begin{split}
{\partial}^{J *} & =  \frac{1}{4m-1} \left(\mathcal{L}^{*}_{k,m} - \mathcal{O}^{*} \right).
\end{split}
\end{equation*}
Thus, it is enough to obtain the adjoints  $\mathcal{L}^{*}_{k,m}$ and   $\mathcal{O}^{*} $ to get ${\partial}^{J *}$, the adjoint of the Jacobi-Serre derivative.  For given integers $k, \ell, m$ and $w$, let 
$$\beta_{k,m,  \ell, w} = \dfrac{m^{2} \left(k-\frac{1}{2}\right) \left(k-\frac{3}{2}\right){(4\ell m-w^{2})^{k-\frac{3}{2}}}}{\pi^{2}}.$$

Let $\mathcal{L}_{k,m}^*$ denote the adjoint of the linear differential operator 
	$\mathcal{L}_{k,m}$, i.e., 
	$$\mathcal{L}_{k,m} ^{*} :  J_{k+2,m}^{cusp} \longrightarrow J_{k,m}^{cusp}.$$
For a given $\phi \in {J}_{k+2,m}^{cusp}$, we characterize the Fourier-Jacobi coefficients of $\mathcal{L}_{k,m} ^{*}\phi$ in terms of special values of certain Dirichlet series in the following theorem. 
\begin{thm}\label{ADLM}
Suppose that $k >4 $ be an integer and $\phi(\tau, z) \in J^{cusp}_{k+2, m}$ with the Fourier expansion
\begin{equation}\label{FJExp}
		\begin{split}
		\phi(\tau, z)=\sum\limits_{\substack{n \ge 1, r\in\mathbb{Z}\\ 4nm-r^{2}>0}} c(n,r)   q^n \zeta^r.
		\end{split}
		\end{equation}
		  Then, $\mathcal{L}_{k,m}^*\phi$ is a Jacobi form of weight $k$ and index $m$ with the Fourier expansion 
$$\mathcal{L}_{k,m}^*\phi(\tau, z)=\sum\limits_{\substack{\ell \ge 1, w~ \in\mathbb{Z}\\ 4\ell m-w^{2}>0}}c_{\mathcal{A}}(\ell,w)   q^{\ell} \zeta^{w}, $$
where 
		%\begin{equation}\label{ACoeff}
		%\begin{split}
		%c_{\mathcal{A}}(\ell,w) = \frac{m^{2} (k-\frac{1}{2})(k-\frac{3}{2})}{{\pi}^{2} (4\ell m-w^{2})^{2}}
		%\begin{cases}
		%\left\{ 
		%\left((4\ell m-w^{2})- \frac{2m(2k-1)}{12}\right) c(\ell, w) \\
		%+ {4m(2k-1)} \displaystyle{\sum_{p=1}^{\infty}} \sigma_{1}(p)c(\ell+p,w) \left(\frac{(4\ell m-w^{2})}{(4(\ell+p) m-w^{2})}\right)^{k+\frac{1}{2}} 
		% \right\}
		%\end{cases}
		%\end{equation}
		\begin{equation*}\label{ACoeff}
		\begin{split}
		c_{\mathcal{A}}(\ell,w) = \beta_{k,m,  \ell, w}
		  \bigg(\dfrac{(4\ell m-w^{2})- \frac{m(2k-1)}{6}}{(4m \ell-w^{2})^{k+\frac{1}{2}}} c(\ell,w)
		    +  4m(2k-1)  \sum_{p=1}^{\infty}    \frac{ \sigma_{1}(p) c(\ell+p,w) }{(4(\ell+p) m-w^{2})^{k+\frac{1}{2}}} \bigg).
		\end{split}
		\end{equation*}
  \end{thm}
Let $\mathcal{O}^{*}$ denote the adjoint of the Oberdieck differential operator $\mathcal{O}$, i.e.,
	\begin{equation*}\label{ObedOperator}
	\begin{split}
	\mathcal{O}^{*}: J_{k+2,m}^{cusp} \rightarrow J_{k, m}^{cusp},  \quad \phi \mapsto  \mathcal{O}^{*} \phi.
	\end{split}
	\end{equation*} 
	It is characterized by the following theorem.

	\begin{thm}\label{ADOBD}
		Suppose that $k > 4$ be an integer  and  $\phi(\tau, z)$ be as in \thmref{ADLM} with $(n,r)$-th Fourier coefficient $c(n,r)$.  Then $\mathcal{O}^{*} \phi$ is a Jacobi form of weight $k$ and index $m$ with the Fourier expansion
$$\mathcal{O}^{*} \phi(\tau, z)=\sum\limits_{\substack{ \ell \ge 1, w \in \mathbb{Z}\\ 4\ell m-w^{2}>0}} c_{\mathcal{B}}(\ell,w)   q^{\ell} \zeta^{w}, $$
		where 
		\begin{equation}\label{BCoeff}
		c_{\mathcal{B}}(\ell,w) = \beta_{k,m,  \ell, w}\; [ A_{1} + A_{2} + A_{31} +A_{32} + A_{4} +A_{5}] 
		\end{equation}
and 
 \begin{equation*}
	\begin{split}
	 A_{1} &= ~~\left( \ell - \frac{k}{12} + \frac{w}{2} +\frac{m}{6} \right)  \frac{c(\ell,w)}{(4\ell m-w^{2})^{k+\frac{1}{2}}},\\
	 A_{2} &=2 k ~ ~ \displaystyle{\sum_{p=1}^{\infty}    \sigma_{1}(p)} ~ \frac{ c(\ell+p,w)}{(4(\ell+p) m-w^{2})^{k+\frac{1}{2}}},\\
	 A_{31}  & = \frac{w}{2}  \sum_{p\ge 0} \frac{c(\ell, w+p+1)}{((4m\ell-(w+p+1)^2))^{k+\frac{1}{2}}} \\
	                    & \!\!\!\! \!\!\!\! \times  \bigg[1- \frac{2 \Gamma(k+1)}{\sqrt{\pi}\Gamma(k+1/2)}  
	                    \frac{(w+p+1)}{((4m\ell-(w+p+1)^2))^{\frac{1}{2}}} {}_{2}F_{1}\left(\frac{1}{2}, k+1; \frac{3}{2}; \frac{-(p+w+1)^{2}}{4 \ell ~m - (p+w+1)^{2}}\right)  \bigg], \\
	 A_{32}  & =  -\frac{w}{2}  \sum_{p\ge 0} \frac{c(\ell, p-w)}{((4m\ell-(p-w)^2))^{k+\frac{1}{2}}} \\
	                    & \quad \times  \bigg[1- \frac{2 \Gamma(k+1)}{\sqrt{\pi}\Gamma(k+1/2)}  
	                    \frac{(p-w)}{((4m\ell-(p-w)^2))^{\frac{1}{2}}} {}_{2}F_{1}\left(\frac{1}{2}, k+1; \frac{3}{2}; \frac{-(p-w)^{2}}{4 \ell ~m - (p-w)^{2}}\right)  \bigg], \\
	A_{4} &=  w~~ \sum_{p \ge 1} \sum_{d \vert p}  \left(  \frac{ c(\ell+p,w+d)}{(4(\ell+p) m-(w+d)^{2})^{k+\frac{1}{2}}} - \frac{ c(\ell+p,w-d)}{(4(\ell+p) m-(w-d)^{2})^{k+\frac{1}{2}}} \right), \\
	A_{5} & =  -2m~~ \sum_{p \ge 1} \sum_{d \vert p} \frac{p}{d} \left(  \frac{ c(\ell+p,w+d)}{(4(\ell+p) m-(w+d)^{2})^{k+\frac{1}{2}}} +  \frac{ c(\ell+p,w-d)}{(4(\ell+p) m-(w-d)^{2})^{k+\frac{1}{2}}} \right) .
\end{split}
\end{equation*}		
%and $A_{i}'s$ are given in \thmref{ADJSerre}. 
%		\begin{equation*}
%	\begin{split}
%	 A_{1} &= ~~\left( \ell - \frac{k}{12} + \frac{w}{2} +\frac{m}{6} \right)  \frac{c(\ell,w)}{(4\ell m-w^{2})^{k+\frac{1}{2}}}.\\
%	 A_{2} &=2 k ~ ~ \displaystyle{\sum_{p=1}^{\infty}    \sigma_{1}(p)} ~ \frac{ c(\ell+p,w)}{(4(\ell+p) m-w^{2})^{k+\frac{1}{2}}}.\\
%	 A_{31}  & = \frac{w}{2}  \sum_{p\ge 0} \frac{c(\ell, w+p+1)}{((4m\ell-(w+p+1)^2))^{k+\frac{1}{2}}} \\
%	                    & \!\!\!\! \!\!\!\! \times  \bigg[1- \frac{2 \Gamma(k+1)}{\sqrt{\pi}\Gamma(k+1/2)}  
%	                    \frac{(w+p+1)}{((4m\ell-(w+p+1)^2))^{\frac{1}{2}}} {}_{2}F_{1}\left(\frac{1}{2}, k+1; \frac{3}{2}; \frac{-(p+w+1)^{2}}{4 \ell ~m - (p+w+1)^{2}}\right)  \bigg] \\
%	 A_{32}  & =  -\frac{w}{2}  \sum_{p\ge 0} \frac{c(\ell, p-w)}{((4m\ell-(p-w)^2))^{k+\frac{1}{2}}} \\
%	                    & \quad \times  \bigg[1- \frac{2 \Gamma(k+1)}{\sqrt{\pi}\Gamma(k+1/2)}  
%	                    \frac{(p-w)}{((4m\ell-(p-w)^2))^{\frac{1}{2}}} {}_{2}F_{1}\left(\frac{1}{2}, k+1; \frac{3}{2}; \frac{-(p-w)^{2}}{4 \ell ~m - (p-w)^{2}}\right)  \bigg] \\
%	A_{4} &=  w~~ \sum_{p \ge 1} \left(\sum_{d \vert p}  \left(  \frac{ c(\ell+p,w+d)}{(4(\ell+p) m-(w+d)^{2})^{k+\frac{1}{2}}} - \frac{ c(\ell+p,w-d)}{(4(\ell+p) m-(w-d)^{2})^{k+\frac{1}{2}}} \right)\right). \\
%	A_{5} & =  -2m~~ \sum_{p \ge 1} \left(\sum_{d \vert p} \frac{p}{d} \left(  \frac{ c(\ell+p,w+d)}{(4(\ell+p) m-(w+d)^{2})^{k+\frac{1}{2}}} +  \frac{ c(\ell+p,w-d)}{(4(\ell+p) m-(w-d)^{2})^{k+\frac{1}{2}}} \right) \!\! \right).
%\end{split}
%\end{equation*}
	\end{thm}

Combining the results of the theorems \ref{ADLM}  and \ref{ADOBD}, we get the Fourier coefficients of  ${\partial}^{J *} \phi $  associated to $\phi \in J^{cusp}_{k+2, m}$ in the following theorem.
\begin{thm}\label{ADJSerre}
Suppose that $k > 4$  be an integer   and  $\phi(\tau, z)$ be as in \thmref{ADLM} with $(n,r)$-th Fourier coefficient $c(n,r)$. Then ${\partial}^{J *} \phi$ is a Jacobi form of weight $k$ and index $m$ with the Fourier expansion 
		$${\partial}^{J *} \phi(\tau, z)=\sum\limits_{\substack{\ell \ge 1, w \in\mathbb{Z}\\ 4\ell m-w^{2}>0}} c_{J}(\ell,w)   q^{\ell} \zeta^{w}, $$
		where 
		$$
		c_{J}(\ell,w) = \frac{\beta_{k, m, \ell,  w}}{4m-1}(A^{*}   - ( A_{1} + A_{2} + A_{31} +A_{32} + A_{4} +A_{5})  )
		$$
\begin{equation*}
		\begin{split}
		&\text{with}~~ \qquad \quad A^{*} =
		  \bigg(\dfrac{(4\ell m-w^{2})- \frac{m(2k-1)}{6}}{(4m \ell-w^{2})^{k+\frac{1}{2}}} c(\ell,w) + ~ 4m(2k-1)  \sum_{p=1}^{\infty}   ~ \frac{ \sigma_{1}(p) c(\ell+p,w) }{(4(\ell+p) m-w^{2})^{k+\frac{1}{2}}} \bigg).
		\end{split}
		\end{equation*}
		and $A_{i}$'s are given in theorem  \thmref{ADOBD}. 
	\end{thm}
	
\section{Key Ingredient}
We know that   the Eisenstein series $E_{2}(\tau)$ satisfies the following transformation property:
\begin{equation}\label{eq:trans of E_2}
%E_{2}|_{2}  \gamma ~ (\tau) = E_{2}(\tau)  +  \frac{6}{ \pi i} \frac{c}{c \tau +d} \quad \text{OR} \quad 
E_{2}\lp \frac {a\tau+b}{c\tau+d}\rp = (c \tau + d)^{2} \left( E_{2}(\tau)+ \frac{6}{\pi i} \frac{c}{c\tau+d}\right)\;\; \text{for all} \abcd \in SL_2(\Z) ~\text{and}~ \tau \in \mathbb{H}, 
\end{equation}	
and it has the following Fourier series given by 
\begin{equation}\label{E2Exp}
\begin{split}
E_{2}(\tau) = 1- 24 \sum_{n \ge 1} \sigma_{1}(n) q^{n}.
\end{split}
\end{equation}
For given positive integers $k, m$ and $n$, an integer $r$ and $\gamma= \left( \left(\begin{array}{cc}
	a & b\\
	c & d\\
	\end{array}\right), (\lambda, \mu) \right) \in \Gamma^J,$ 
let 
\begin{equation}\label{AEq}
\begin{split}
A_{k,m, n,r}(\gamma .(\tau,z)) &:= {(q^{n} \zeta^{r})}|_{k, m}  \gamma  \\ 
&= (c\tau+d)^{-k} 
e^{ 2 \pi i m\left(-\dfrac{c(z+\lambda \tau+\mu)^2}{c\tau+d}+ \lambda^2\tau+2\lambda z\right)}
e^{2\pi i \Big( n\frac {a\tau+b}{c\tau+d}+ r\frac{z+\lambda \tau+ \mu}{c\tau+d} \Big)}.
\end{split}
\end{equation} 
We observe that 
\begin{equation}\label{AEq}
\begin{split}
 A_{k+2,m,n,r}(\gamma .(\tau,z)) = (c\tau+d)^{-2} A_{k,m,n,r}(\gamma .(\tau,z)).
 \end{split}
\end{equation}
The derivatives of the function $A_{k,m, n,r}(\gamma .(\tau,z))$ also satisfy the following properties.
\begin{equation}\label{eq:D_tau}
\begin{split}
&D_{\tau} A_{k,m,n,r}(\gamma .(\tau,z)) = A_{k,m,n,r}(\gamma .(\tau,z))
\Bigg[
\frac{-kc}{2\pi i(c\tau+d)} -\frac{2mc\l(z+\l\t+\mu)}{c\t+d}  
\\
&\hspace{2.5cm}
+\frac{mc^2(z+\l\t+\mu)^2}{(c\t+d)^2} +m\l^2+\frac{n}{(c\t +d)^2} +\frac{r\l}{c\t+d} -\frac{rc(z+\l\t+\mu)}{(c\t+d)^2}
\Bigg],
\end{split}
\end{equation}
\begin{equation}\label{eq:D_z}
\begin{split}
&D_{z}A_{k,m, n,r}(\gamma .(\tau,z))= A_{k,m, n,r}(\gamma .(\tau,z))
\Bigg[
\frac{r-2mc(z+\l\t+\mu)}{c\t+d}
+2m\l
\Bigg]
\end{split}
\end{equation}
and
\begin{equation}\label{eq:D_z^2}
\begin{split}
&D_{z}^2
A_{k,m, n,r}(\gamma .(\tau,z)) = A_{k,m, n,r}(\gamma .(\tau,z)) \Bigg[
\Bigg(
\frac{r -2mc(z+\l\t+\mu)}{c\t+d}
+2m\l
\Bigg)^2
\!\!\!-\frac{mc}{\pi i(c\tau+d)} 
\Bigg].
\end{split}
\end{equation}
The above transformation properties helps  us to get the transformation property of the action of $\mathcal{L}_{k,m}$, Oberdieck derivative $\mathcal{O}$ and Jacobi-Serre derivative ${\partial}^{J}$ on the monomials. For the sake of completeness, we provide brief details of the proof.

	\begin{lem}\label{StokeHeat}
		For any $k, m, n \in \mathbb{N}$, $r \in \mathbb{Z}$ and any $\gamma= \left( \left(\begin{array}{cc}
	a & b\\
	c & d\\
	\end{array}\right), (\lambda, \mu) \right) \in \Gamma^J,$ we have 
		\begin{equation}\label{StokeHeatAction}
		\left(\mathcal{L}_{m}- \frac{m(2k-1)}{6}E_{2}\right) ({q^{n} \zeta^{r}}|_{k, m}  \gamma) = \left( \left(\mathcal{L}_{m}- \frac{m(2k-1)}{6}E_{2}\right) {q^{n} \zeta^{r}}\right)|_{k+2, m}  \gamma. 
		\end{equation}
	\end{lem}
	\noindent
	\begin{proof}
%		To prove the lemma, we use the following transformation law
%		\begin{equation}\label{eq:trans law of E_2}
%		E_2 \Big( \frac{a \tau + b}{c \tau + d} \Big) 
%		= (c \tau + d)^2 E_2(\tau)
%		+ \frac{6}{\pi i} c (c \tau + d),
%		\end{equation}
Let us consider
	$\phi(\t, z)
	:=\lp \mathcal{L}_{m}- \frac{m(2k-1)}{6}E_{2}\rp q^n\z^r 
	= \lp 4m D_{\tau} -D_z^2  - \frac{m(2k-1)}{6} E_{2} \rp q^n\z^r.$ 
 %we find
%\[\phi(\t, z)=\lp 4m n -r^2  - \frac{m(2k-1)}{6} E_{2} \rp q^n\z^r;\]
Then,
\begin{equation}\label{eq: 1nd part of L_m ... and slash}
\begin{split}
&\phi|_{k+2, m} \gamma  =  A_{k+2,m, n,r}(\gamma .(\tau,z))\times \bigg[ 4m n -r^2  - \frac{m(2k-1)}{6} E_{2}\lp \frac {a\tau+b}{c\tau+d}\rp \bigg] \\
& =  A_{k+2,m, n,r}(\gamma .(\tau,z))\times \bigg[ 4m n -r^2  - \frac{m(2k-1)}{6} (c \tau+d)^{2} \left( E_{2}(\tau)+ \frac{6c}{ \pi i (c\tau+d)} \right) \bigg] \\ 
%&= (c\tau+d)^{-k-2} e^{ 2 \pi i m\left(-\dfrac{c(z+\lambda \tau+\mu)^2}{c\tau+d}+ \lambda^2\tau+2\lambda z\right)} \phi \left( \frac {a\tau+b}{c\tau+d}, \frac{z+\lambda \tau+ \mu}{c\tau+d} \right) \\ 
%&=  (c\tau+d)^{-k-2} e^{ 2 \pi i m\left(-\dfrac{c(z+\lambda \tau+\mu)^2}{c\tau+d}+ \lambda^2\tau+2\lambda z\right)}e^{2\pi i \Big( n\frac {a\tau+b}{c\tau+d}+ r\frac{z+\lambda \tau+ \mu}{c\tau+d} \Big)}\\
%&\hspace{5cm} \times \bigg[ 4m n -r^2  - \frac{m(2k-1)}{6} E_{2}\lp \frac {a\tau+b}{c\tau+d}\rp \bigg]
\end{split}
\end{equation}
which follows from \eqref{eq:trans of E_2} and  \eqref{AEq}. % the transformation property for $E_{2}(\tau)$ given in \eqref{eq:trans of E_2} 
%Using \eqref{eq:trans law of E_2}, we get
%	\begin{equation}\label{eq: 1st part of L_m ... and slash}
%		\begin{split}
%	&\quad   
%	\phi|_{k+2, m} \gamma
%	=  
%	(c\tau+d)^{-k-2} 
%	e^{ 2 \pi i m\left(-\dfrac{c(z+\lambda \tau+\mu)^2}{c\tau+d}+ \lambda^2\tau+2\lambda z\right)}
%	e^{2\pi i \Big( n\frac {a\tau+b}{c\tau+d}+ r\frac{z+\lambda \tau+ \mu}{c\tau+d} \Big)}\\
%	&\hspace{5cm} \times \bigg[ 4m n -r^2  - \frac{m(2k-1)}{6}(c\tau+d)^2 E_{2}-\frac{m(2k-1)}{\pi i}c(c\tau+d) \bigg].
%		\end{split}
%	\end{equation}
Now, we consider
\begin{equation*}
\begin{split}
&\left(\mathcal{L}_{m}- \frac{m(2k-1)}{6}E_{2}\right) ({q^{n} \zeta^{r}}|_{k, m}  \gamma) = \lp 4m D_{\tau} -D_z^2  - \frac{m(2k-1)}{6} E_{2} \rp A_{k,m, n,r}(\gamma .(\tau,z)).  \\	
%&= \lp 4m D_{\tau} -D_z^2  - \frac{m(2k-1)}{6} E_{2} \rp\\
%&\hspace{3cm}
%\Bigg(
%(c\tau+d)^{-k} 
%e^{ 2 \pi i m\left(-\dfrac{c(z+\lambda \tau+\mu)^2}{c\tau+d}+ \lambda^2\tau+2\lambda z\right)}
%e^{2\pi i \Big( n\frac {a\tau+b}{c\tau+d}+ r\frac{z+\lambda \tau+ \mu}{c\tau+d} \Big)}
%\Bigg).
\end{split}	
\end{equation*}
	Using  \eqref{eq:trans of E_2}, \eqref{AEq}, \eqref{eq:D_tau} and \eqref{eq:D_z^2},  we get
\begin{equation}\label{eq: 2nd part of L_m ... and slash}
\begin{split}
&\left(\mathcal{L}_{m}- \frac{m(2k-1)}{6}E_{2}(\tau)\right) ({q^{n} \zeta^{r}}|_{k, m}  \gamma) \\
&  = A_{k,m, n,r}(\gamma .(\tau,z))
\Bigg[
\frac{4mn-r^{2}}{(c\t+d)^2}
-\frac{m(2k-1)}{6}\Big(E_{2}(\tau)
+\frac{6c}{\pi i(c\t+d)}\Big)	
\Bigg] \\
&  = A_{k+2,m, n,r}(\gamma .(\tau,z)) \bigg[ 4m n -r^2  - \frac{m(2k-1)}{6} (c \tau+d)^{2} \left( E_{2}(\tau)+ \frac{6c}{ \pi i (c\tau+d)} \right) \bigg].\\
%&= 
%(c\tau+d)^{-k} 
%e^{ 2 \pi i m\left(-\dfrac{c(z+\lambda \tau+\mu)^2}{c\tau+d}+ \lambda^2\tau+2\lambda z\right)}
%e^{2\pi i \Big( n\frac {a\tau+b}{c\tau+d}+ r\frac{z+\lambda \tau+ \mu}{c\tau+d} \Big)}
%\Bigg[
%\frac{-4mkc}{2\pi i(c\tau+d)}+\frac{4mn}{(c\t +d)^2}\\
%&\hspace{1cm}
%+\frac{4mr\l}{c\t+d}
%-\frac{4mrc(z+\l\t+\mu)}{(c\t+d)^2}
%+\frac{-8m^2c\l(z+\l\t+\mu)}{c\t+d}
%+\frac{-8m^2c\l(z+\l\t+\mu)}{c\t+d}	
%\\
%&\hspace{1cm}
%+\frac{4m^2c^2(z+\l\t+\mu)^2}{(c\t+d)^2} 
%+4m^2\l^2-	\Bigg(
%\frac{-2mc(z+\l\t+\mu)}{c\t+d}
%+2m\l
%+\frac{r}{c\t+d}
%\Bigg)^2
%+\frac{mc}{\pi i(c\tau+d)}\\ 
%&\hspace{1cm}
%-\frac{m(2k-1)}{6}E_{2}
%\Bigg]\\
%&=
%(c\tau+d)^{-k} 
%e^{ 2 \pi i m\left(-\dfrac{c(z+\lambda \tau+\mu)^2}{c\tau+d}+ \lambda^2\tau+2\lambda z\right)}
%e^{2\pi i \Big( n\frac {a\tau+b}{c\tau+d}+ r\frac{z+\lambda \tau+ \mu}{c\tau+d} \Big)}\\
%&	\hspace{5cm}
%\times
%\Bigg[
%\frac{4mn}{(c\t+d)^2}
%-\frac{r^2}{(c\t+d)^2}
%-\frac{m(2k-1)}{6}\Big(E_2
%+\frac{6c}{\pi i(c\t+d)}\Big)	
%\Bigg].
\end{split}	
\end{equation}
From \eqref{eq: 1nd part of L_m ... and slash} and \eqref{eq: 2nd part of L_m ... and slash}, we get the required transformation property.	
\end{proof}
	
\begin{lem}\label{StokeOBD}
For any $k, m, n \in \mathbb{N}$, $r \in \mathbb{Z}$ and any $\gamma= \left( \left(\begin{array}{cc}
	a & b\\
	c & d\\
	\end{array}\right), (\lambda, \mu) \right) \in \Gamma^J,$ we have  
\begin{equation*}\label{StokeOBDAction}
\begin{split}
&    \left(  \left(  D_{\tau}  - \frac{k}{12} E_{2}  - J_{1} D_{z}  + m J_{2} \right) {q^{n} \zeta^{r}} \right)|_{k+2, m}  \gamma = \left(  D_{\tau}  - \frac{k}{12} E_{2}  - J_{1} D_{z}  + m J_{2}\right) ({q^{n} \zeta^{r}}|_{k, m}  \gamma). \\
\end{split}
\end{equation*}
\end{lem}
\begin{proof}
We prove the lemma using the following transformation properties obtained from \eqref{J1Transf} and \eqref{J2Transf}:
\begin{equation}\label{eq:trans law of J_1}
J_1\lp \frac {a\tau+b}{c\tau+d}, \frac{z+\lambda \tau+ \mu}{c\tau+d} \rp
=(c\t+d)\Big(J_1(\t, z)	-\l\Big) +c(z+\l\t+\mu),
\end{equation}
and
\begin{equation}\label{eq:trans law of J_2}
\begin{split}
J_2\lp \frac {a\tau+b}{c\tau+d}, \frac{z+\lambda \tau+ \mu}{c\tau+d} \rp
&=(c\t+d)^2\Big(J_2(\t, z)-2\l J_1(\t, z)+\l^2\Big)\\
&+2c(c\t+d)(z+\l\t+\mu)\Big( J_1(\t,z) -\l\Big) +c^2(z+\l\t+\mu)^2.
\end{split}
\end{equation}
Let $\phi(\t, z): = \lp D_{\tau}  - \frac{k}{12} E_{2}  - J_{1} D_{z}  + m J_{2} \rp q^n\z^r =\left(n-\frac{k}{12} E_2-rJ_1+mJ_2\right)q^n\z^r$. 
%we find
%\[
%\phi(\tau, z) 
%=\left(n-\frac{k}{12} E_2-rJ_1+mJ_2\right)q^n\z^r;
%\]  
Then, %Therefore,
\begin{align*} 
&\quad   
\phi|_{k+2, m} \gamma = A_{k+2,m, n,r}(\gamma .(\tau,z))\Bigg[
n-\frac{k}{12} E_2\lp \frac {a\tau+b}{c\tau+d}\rp
-rJ_1\left( \frac {a\tau+b}{c\tau+d}, \frac{z+\lambda \tau+ \mu}{c\tau+d} \right)\\
&\hspace{9 cm}
+mJ_2 \left( \frac {a\tau+b}{c\tau+d}, \frac{z+\lambda \tau+ \mu}{c\tau+d} \right)
\Big)
\Bigg]. \\ 
%&=  
%(c\tau+d)^{-k-2} 
%e^{ 2 \pi i m\left(-\dfrac{c(z+\lambda \tau+\mu)^2}{c\tau+d}+ \lambda^2\tau+2\lambda z\right)}
%\phi 
%\left( \frac {a\tau+b}{c\tau+d}, \frac{z+\lambda \tau+ \mu}{c\tau+d} %\right)
%\\ 
%&=  
%(c\tau+d)^{-k-2} 
%e^{ 2 \pi i m\left(-\dfrac{c(z+\lambda \tau+\mu)^2}{c\tau+d}+ \lambda^2\tau+2\lambda z\right)}
%e^{2\pi i \Big( n\frac {a\tau+b}{c\tau+d}+ r\frac{z+\lambda \tau+ \mu}{c\tau+d} \Big)}
%\\
%&\hspace{2cm}
%\times
%\Bigg[
%n-\frac{k}{12} E_2\lp \frac {a\tau+b}{c\tau+d}\rp
%-rJ_1\left( \frac {a\tau+b}{c\tau+d}, \frac{z+\lambda \tau+ \mu}%{c\tau+d} \right)\\
%&\hspace{2.9cm}
%+mJ_2 \left( \frac {a\tau+b}{c\tau+d}, \frac{z+\lambda \tau+ \mu}{c\tau+d} \right)
%\Big)
%\Bigg].
\end{align*} 
Using \eqref{eq:trans of E_2}, \eqref{eq:trans law of J_1} and \eqref{eq:trans law of J_2}, we get
\begin{equation}\label{eq: 1st part of D_tau ... and slash}
\begin{split}
&\quad   
\phi|_{k+2, m} \gamma = A_{k+2,m, n,r}(\gamma .(\tau,z)) \times
\Bigg[
n- \frac{k}{12} 
\Big((c\t +d)^2E_2(\tau)
\quad +\frac{6c}{\pi i}(c\t+d)\Big) 
\\ &\hspace{2cm}
-r (c\t+d)\Big(J_1	-\l\Big)
-rc(z+\l\t+\mu)
+ m(c\t+d)^2\Big(J_2-2\l J_1+\l^2\Big)
 \\
&\hspace{4.2cm}
+2mc(c\t+d)(z+\l\t+\mu)\Big( J_1 -\l\Big)
+mc^2(z+\l\t+\mu)^2
\Bigg].	 
\\ 
%&=  
%(c\tau+d)^{-k-2} 
%e^{ 2 \pi i m\left(-\dfrac{c(z+\lambda \tau+\mu)^2}{c\tau+d}+ %\lambda^2\tau+2\lambda z\right)}
%e^{2\pi i \Big( n\frac {a\tau+b}{c\tau+d}+ r\frac{z+\lambda \tau+ \mu}{c\tau+d} \Big)}\\
%&	\hspace{2cm}
%\times
%\Bigg[
%n- \frac{k}{12} 
%\Big((c\t +d)^2E_2
%\quad +\frac{6c}{\pi i}(c\t+d)\Big)
%-r (c\t+d)\Big(J_1	-\l\Big)
%\\
%&\hspace{3cm}
%-rc(z+\l\t+\mu)
%+
%m(c\t+d)^2\Big(J_2-2\l J_1+\l^2\Big)
%+2mc(c\t+d)(z+\l\t+\mu)\Big( J_1 -\l\Big) \\
%&\hspace{3cm}
%+mc^2(z+\l\t+\mu)^2
%\Bigg].	
\end{split}
\end{equation}
Now, we consider $G:= \left( D_{\tau}  - \frac{k}{12} E_{2}(\tau)  - J_{1} D_{z}  + m J_{2}\right) ({q^{n} \zeta^{r}}|_{k, m}  \gamma)$. Using \eqref{AEq}, we have
\begin{align*}
%&\left( D_{\tau}  - \frac{k}{12} E_{2}  - J_{1} D_{z}  + m J_{2}\right) ({q^{n} \zeta^{r}}|_{k, m}  \gamma) 
G = \left( D_{\tau}  - \frac{k}{12} E_{2}  - J_{1} D_{z}  + m J_{2}\right) A_{k,m, n,r}(\gamma .(\tau,z)).\\	
%&= 
%\Big(
%D_{\t}-\frac{k}{12}E_2-J_1D_z+mJ_2
%\Big)\\
%&\hspace{3cm}
%\Bigg(
%(c\tau+d)^{-k} 
%e^{ 2 \pi i m\left(-\dfrac{c(z+\lambda \tau+\mu)^2}{c\tau+d}+ %\lambda^2\tau+2\lambda z\right)}
%e^{2\pi i \Big( n\frac {a\tau+b}{c\tau+d}+ r\frac{z+\lambda \tau+ \mu}{c\tau+d} \Big)}
%\Bigg).
\end{align*}
Using \eqref{eq:D_tau} and \eqref{eq:D_z} , we get
\begin{equation}\label{eq: 2nd part of D_tau ... and slash}
\begin{split}
%&\left( D_{\tau}  - \frac{k}{12} E_{2}  - J_{1} D_{z}  + m J_{2}\right) ({q^{n} \zeta^{r}}|_{k, m}  \gamma)\\
%&G = A_{k,m, n,r}(\gamma .(\tau,z))
%(c\tau+d)^{-k} 
%e^{ 2 \pi i m\left(-\dfrac{c(z+\lambda \tau+\mu)^2}{c\tau+d}+ \lambda^2\tau+2\lambda z\right)}
%e^{2\pi i \Big( n\frac {a\tau+b}{c\tau+d}+ r\frac{z+\lambda \tau+ \mu}{c\tau+d} \Big)}
%\Bigg[
%\frac{-kc}{2\pi i(c\tau+d)}+\frac{n}{(c\t +d)^2}\\
%&\hspace{1cm}
%+\frac{r\l}{c\t+d}
%-\frac{rc(z+\l\t+\mu)}{(c\t+d)^2}
%+\frac{-2mc\l(z+\l\t+\mu)}{c\t+d}
%+\frac{-2mc\l(z+\l\t+\mu)}{c\t+d}
%+\frac{mc^2(z+\l\t+\mu)^2}{(c\t+d)^2} 
%\\
%&\hspace{1cm}
%+m\l^2	
%-\frac{k}{12}E_2
%-J_1\Big(
%\frac{-2mc(z+\l\t+\mu)}{c\t+d}
%+2m\l
%+\frac{r}{c\t+d}
%\Big)
%+mJ_2
%\Bigg]\\
&G = A_{k,m, n,r}(\gamma .(\tau,z))
%(c\tau+d)^{-k} 
%e^{ 2 \pi i m\left(-\dfrac{c(z+\lambda \tau+\mu)^2}{c\tau+d}+ \lambda^2\tau+2\lambda z\right)}
%e^{2\pi i \Big( n\frac {a\tau+b}{c\tau+d}+ r\frac{z+\lambda \tau+ \mu}{c\tau+d} \Big)}\\
%&	\hspace{1cm}
\Bigg[
\frac{n}{(c\t+d)^2}-\frac{k}{12}
\Big(E_2
+\frac{6c}{\pi i(c\t+d)}\Big)	
-\frac{r}{c\t+d} \Big(J_1	-\l\Big)
-rc\frac{z+\l\t+\mu}{(c\t+d)^2}
\\
&\hspace{2.8cm}
+
m\Big(J_2-2\l J_1+\l^2\Big)
+2mc\frac{z+\l\t+\mu}{c\t+d}\Big( J_1 -\l\Big) 
+mc^2\lp\frac{z+\l\t+\mu}{c\t+d}\rp^2
\Bigg)
\Bigg].
\end{split}	
\end{equation}
The assertion of the lemma follows from 
 \eqref{eq: 1st part of D_tau ... and slash} and \eqref{eq: 2nd part of D_tau ... and slash}.
\end{proof}
From \eqref{Jacobi-Serre-DSumT}, we know that Jacobi-Serre derivative is given as follows:
$$
{\partial}^{J} = \frac{1}{1-4m} \left[ \left(-\mathcal{L}_{m}+ \frac{m(2k-1)}{6}E_{2}\right) +  \left(   D_{\tau}  - \frac{k}{12} E_{2}  - J_{1} D_{z}  + m J_{2} \right) \right].
$$
%The required functional equation for  Jacobi-Serre derivative ${\partial}^{J}$ follows from 
From Lemmas \ref{StokeHeat} and \ref{StokeOBD}, we obtain that the action of Jacobi -Serre derivative and the action of stoke operator commute on monomials. In particular, we have the following lemma.
	\begin{lem}\label{StokeJacobiS}
		For any $k, m, n \in \mathbb{N}$, $r \in \mathbb{Z}$, and any $\gamma= \left( \left(\begin{array}{cc}
	a & b\\
	c & d\\
	\end{array}\right), (\lambda, \mu) \right) \in \Gamma^J,$ we have 
		\begin{equation*}
		 {\partial}^{J} ({q^{n} \zeta^{r}}|_{k, m}  \gamma) = \left( {\partial}^{J}  {q^{n} \zeta^{r}}\right)|_{k+2, m}  \gamma. 
		\end{equation*}
	\end{lem}
\noindent	
We prove our next lemma which plays a crucial role to prove the main result of the paper.  
\begin{lem}\label{IntegralApprox}
Let $k,m,  \ell$ and $w$ be the integers and $4 \ell m- w^{2}>0$. Let  $q=e^{2 \pi i \tau}$ and $\zeta=e^{2 \pi i z}$ with $\tau = u+i v \in \mathbb{H}$ and $z= x+iy \in \mathbb{C}$, and $ dV_J := \frac{du dv dx dy}{v^{3}}$ denote the measure on $\mathbb{H} \times \mathbb{C}$. Then, we have  
\begin{equation*}
\begin{split}
%  \int_{0}^{1} \int_{0}^{\infty} \int_{0}^{1} \int_{-\infty}^{\infty} 
I :=  \int \limits_{\Gamma_{\infty}^J\setminus\mathbb{H}\times \mathbb{C}} q^{n} \zeta^{r} \overline{q^{\ell} \zeta^{w}}~ v^{k+2} e^{\frac{-4 \pi my^{2}}{v}}  dV_J = %\frac{du dv dx dy}{v^{3}}  = 
\begin{cases} 
\frac{m^{k}}{ 2 {\pi}^{k+\frac{1}{2}}} ~ \frac{\Gamma\left(k+\frac{1}{2}\right)}{(4\ell m-w^{2})^{k+\frac{1}{2}}} \quad {\rm if} \quad n= \ell ~\& ~r=w, \\
0 \qquad \qquad \qquad \qquad \qquad {\rm otherwise}.
\end{cases}
\end{split}
\end{equation*}
\end{lem}
%\noindent
%\textbf{Proof:} 
\begin{proof}
 Substituting $\tau = u+i v$, $z= x+iy$ and $dV_J = \frac{du dv dx dy}{v^{3}}$, we get
\begin{equation*}
\begin{split}
I & =  \int \limits_{\Gamma_{\infty}^J\setminus\mathbb{H}\times \mathbb{C}} q^{n} \zeta^{r} \overline{q^{\ell} \zeta^{w}}~ v^{k+2} e^{\frac{-4 \pi my^{2}}{v}}  dV_J \\
& =   \int_{0}^{1} \int_{0}^{\infty} \int_{0}^{1} \int_{-\infty}^{\infty} 
e^{-2\pi v(n+\ell)} e^{2\pi i(n- \ell)u} e^{-2\pi y(r+w)}e^{2\pi i(r-w)x}
v^{k+2}e^{\dfrac{-4\pi my^2}{v}} \frac{du dv dx dy}{v^{3}}.
\end{split}
\end{equation*}
Integrating over $x$ and $u$, the integral survives only if $ n= \ell$ and $r=w$ otherwise it vanishes. Thus,  we obtain
\begin{equation*}
\begin{split}
I & =  \int_{0}^{\infty} \!\!\!\! \int_{-\infty}^{\infty} 
e^{-4\pi \ell v}e^{-4\pi w y} v^{k+2}e^{\dfrac{-4\pi my^2}{v}} \frac{ dv dy}{v^{3}} =  \int_{0}^{\infty} \!\! \left[\int_{-\infty}^{\infty} 
e^{-4\pi \left(w y+\frac{my^2}{v} \right)} dy  \right] e^{-4\pi \ell v} v^{k+2} \frac{ dv}{v^{3}}.
\end{split}
\end{equation*}
Using the integral estimate $\displaystyle{\int_{-\infty}^{\infty}}e^{-4\pi\left(ry+\frac{my^{2}}{v}\right)}dy = \dfrac{\sqrt{v}}{2 \sqrt{m}}~ e^{\pi\dfrac{r^{2}v}{m}}$, we obtain  
\begin{equation*}
\begin{split}
I &  = \int_{0}^{\infty}  \!\!\!\! \dfrac{\sqrt{v}}{2 \sqrt{m}} e^{\pi\dfrac{w^{2}v}{m}} e^{-4\pi \ell v} v^{k+2} \frac{ dv}{v^{3}} 
=  \frac{1}{2 \sqrt{m}}  \int_{0}^{\infty} \!\! e^{-4\pi (\ell - \frac{w^{2}}{4m})  v} v^{k+\frac{1}{2}} \frac{ dv}{v} 
=\frac{m^{k}}{ 2 {\pi}^{k+\frac{1}{2}}} ~ \frac{\Gamma\left(k+\frac{1}{2}\right)}{(4\ell m-w^{2})^{k+\frac{1}{2}}}.
\end{split}
\end{equation*}
The last integral is obtained by substituting $t = 4 \pi(\ell- \frac{w^{2}}{4m})v$ and using the Gamma function $\Gamma(s) = \displaystyle{\int_0^{\infty} } e^{-t} t^{s} \frac{dt}{t}$ at $s = k+\frac{1}{2}$.% (or see \eqref{jacobi-intv} by replacing $k$ by $k+2$). 
\end{proof}
	
%We need the following lemma to prove the result.
\begin{lem}\label{jacobi-mainlemma}
	Suppose that $k > 4$ be an integer  and $\phi(\tau, z) \in J^{cusp}_{k+2, m}$ with $(n,r)$-th Fourier coefficient $c(n,r)$. Let $\ell \in  \mathbb{N}$ and $w \in  \mathbb{Z}$ such that $4 \ell m-w^{2}>0$. Then the integrals  
		\begin{equation*}
		\begin{split}
		& I_{1}=\!\!\sum_{\gamma\in\Gamma_{\infty}^J\setminus \Gamma^J} 
		\int\limits_{\Gamma^J\setminus\mathbb{H}\times \mathbb{C}}  
		\vt{\phi(\tau,z) \overline{ \left(\mathcal{L}_{m}- \frac{2m(2k-1)}{12}E_{2}\right) ({q^{\ell} \zeta^{w}}|_{k, m} \gamma)}v^{k+2}e^{\dfrac{-4\pi my^2}{v}}}  dV_J \\
\text{and} \quad 	& I_{2} =\!\!	\sum_{\gamma\in\Gamma_{\infty}^J\setminus \Gamma^J} 
		\!\int\limits_{\Gamma^J\setminus\mathbb{H}\times \mathbb{C}}  
		\vt{\phi(\tau,z) \overline{  \left(  D_{\tau}  - \frac{k}{12} E_{2}  - J_{1} D_{z}  + m J_{2}\right) ({q^{\ell} \zeta^{w}}|_{k, m} \gamma)}v^{k+2}e^{\dfrac{-4\pi my^2}{v}}}  dV_J
		\end{split}
		\end{equation*}
		converge absolutely. 
	\end{lem}
	
	\begin{proof}
	Let us consider the first integral $I_{1}$. Changing the variable $(\tau, z)$ to $\gamma^{-1}.(\tau, z)$ and using \eqref{StokeHeatAction}, the sum equals to
		\begin{equation*}
		I_{1}= \sum_{\gamma\in\Gamma_{\infty}^J\setminus \Gamma^J} 
		\int\limits_{\gamma.\Gamma^J\setminus\mathbb{H}\times \mathbb{C}}  
		\vt{ \phi(\tau,z) \overline{  \left(\mathcal{L}_{m}- \frac{2m(2k-1)}{12}E_{2}\right) ({q^{\ell} \zeta^{w}})|_{k+2, m} \gamma}~v^{k+2}e^{\dfrac{-4\pi my^2}{v}}} dV_J.
		\end{equation*}
		Now, using Rankin unfolding argument, we see that
		$$
	I_{1}=	\int\limits_{\Gamma^J_{\infty}\setminus\mathbb{H}\times \mathbb{C}} \vt{ \phi(\tau,z) \overline{ \left(\mathcal{L}_{m}- \frac{2m(2k-1)}{12}E_{2}\right) ({q^{\ell} \zeta^{w}})|_{k+2, m} \gamma}~v^{k+2}e^{\dfrac{-4\pi my^2}{v}}} dV_J.
		$$
		Now, replacing $\phi(\tau,z)$ and $E_{2}(\tau)$ with their Fourier expansions and the action of the operator $\mathcal{L}_{m}$ on ${q^{n} \zeta^{r}}$,  the last integral is majorized by 
		\begin{equation*}
		\begin{split}
		S & = \sum\limits_{\substack{n \ge 1, r \in\mathbb{Z}\\ 4nm-r^{2}>0}} c(n,r) \left((4nm-r^{2})- \frac{2m(2k-1)}{12}\right)  \int\limits_{\Gamma^J_{\infty}\setminus\mathbb{H}\times \mathbb{C}} {q^{n} \zeta^{r}}  \overline{{q^{\ell} \zeta^{w}}}~v^{k+2}e^{\dfrac{-4\pi my^2}{v}} dV_J \\
		& \quad + ~ 4m(2k-1) \sum\limits_{\substack{n \ge 1, r \in\mathbb{Z}\\ 4nm-r^{2}>0}} c(n,r) \displaystyle{\sum_{p=1}^{\infty} \sigma_{1}(p)} \int\limits_{\Gamma^J_{\infty}\setminus\mathbb{H}\times \mathbb{C}} {q^{n} \zeta^{r}}  \overline{{q^{\ell + p} \zeta^{w}}}~v^{k+2}e^{\dfrac{-4\pi my^2}{v}} dV_J.
		\end{split}
		\end{equation*}
		Using \lemref{IntegralApprox}, we have 
		\begin{equation*}
		\begin{split}
		S & =  c(\ell,w) \left((4\ell m-w^{2})- \frac{2m(2k-1)}{12}\right)   \frac{m^{k}}{ 2 {\pi}^{k+\frac{1}{2}}} ~ \frac{\Gamma\left(k+\frac{1}{2}\right)}{(4\ell m-w^{2})^{k+\frac{1}{2}}}  \\
		& \quad + ~ 4m(2k-1)  \displaystyle{\sum_{p=1}^{\infty}   c(\ell+p,w) \sigma_{1}(p)}  \frac{m^{k}}{ 2 {\pi}^{k+\frac{1}{2}}} ~ \frac{\Gamma\left(k+\frac{1}{2}\right)}{(4(\ell+p) m-w^{2})^{k+\frac{1}{2}}}.
		\end{split}
		\end{equation*}
		\begin{equation*}
		\begin{split}
		S & =   \frac{m^{k}{\Gamma\left(k+\frac{1}{2}\right)}}{ 2 {\pi}^{k+\frac{1}{2}}} \bigg[ 
		  \frac{c(\ell,w)}{(4\ell m-w^{2})^{k+\frac{1}{2}}} \left((4\ell m-w^{2})- \frac{2m(2k-1)}{12}\right)  \\ & \hspace{3.5 cm}
		 +  4m(2k-1)  \displaystyle{\sum_{p=1}^{\infty}    \sigma_{1}(p)}  \frac{ c(\ell+p,w)}{(4(\ell+p) m-w^{2})^{k+\frac{1}{2}}}  \bigg].
		\end{split}
		\end{equation*}
Using the growth of the Fourier coefficients given in \lemref{Jacobi-convergence}, the above series converges absolutely for $k>4$. Similarly, the absolute convergence of $I_{2}$ can be easily seen.  This completes  the proof. 
	\end{proof}
	
\section{Proof of Results}
	\noindent
	\textbf{Proof of \thmref{ADLM}:}
Let $\phi(\tau, z) \in J^{cusp}_{k+2, m}$ with $(n,r)$-th  Fourier coefficient $c(n,r)$ given in  the Fourier expansion \eqref{FJExp}.
%$$\phi(\tau, z)=\sum\limits_{\substack{n \ge 1, r\in\mathbb{Z}\\ 4nm-r^{2}>0}} c(n,r)   q^n \zeta^r.$$ 
 Let $\mathcal{L}^{*}_{k,m}$ denote the adjoint of $\mathcal{L}_{k,m}$.  Let the Fourier expansion of   $\mathcal{L}_{k,m}^*\phi$ is given by
$$
	(\mathcal{L}_{k,m}^*\phi) (\tau, z)=  \sum_{\substack {\ell \ge 1, w \in \mathbb{Z}, \\ 4m \ell - w^{2} > 0}} c_{\mathcal{A}}(\ell, w) q^{\ell} \zeta^{w}. 
	$$
	Now, we consider the $(\ell, w)$-th Poinca{\'r}e series of weight $k$ and index $m$ as given in \eqref{jacobi-poincare}. Using \lemref{jacobi-poincare-lemma}, we have 
	$$
	\langle \mathcal{L}_{k,m}^*\phi,  P_{k,m;(\ell,w)}\rangle=  \dfrac{m^{k-2}\Gamma(k-\frac{3}{2})}{2 \pi^{k-\frac{3}{2}}} (4m \ell -w^2)^{-k+\frac{3}{2}} c_{\mathcal{A}}(\ell,w).
	$$
	%where 
	%$$
	%\alpha_{k,m} = \dfrac{m^{k-2}\Gamma(k-\frac{3}{2})}{2 \pi^{k-\frac{3}{2}}}.
	%$$
	On the other hand, by definition of the adjoint map we have
	$$
	\langle  \mathcal{L}_{k,m}^*\phi,  P_{k,m;(\ell,w)}\rangle =  \pangle {\phi,   \left( \mathcal{L}_{m}- \frac{m(2k-1)}{6}E_{2} \right) P_{k,m;(\ell,w)} }.
	$$
	Hence, we get 
	\begin{equation}\label{AJacobiCoeff}
	c_{\mathcal{A}}(\ell,w) = \dfrac {2 \pi^{k-\frac{3}{2}}}{m^{k-2}\Gamma(k-\frac{3}{2})} {(4m \ell-w^{2})^{k-\frac{3}{2}}} ~~  \pangle{ \phi,   \left( \mathcal{L}_{m}- \frac{m(2k-1)}{6}E_{2} \right) P_{k,m;(\ell,w)}}.
	\end{equation}
Let $S_{1}$ denote the Petersson Scalar product 	$\pangle{ \phi,   \left( \mathcal{L}_{m}- \frac{m(2k-1)}{6}E_{2} \right) P_{k,m;(\ell,w)}}$. By definition of  Petersson Scalar product, we have
%By definition, $\pangle{ \phi,   \left( \mathcal{L}_{m}- \frac{m(2k-1)}{6}E_{2} \right) P_{k,m;(\ell,w)}}$ equals to $S_1$ (say)
	\begin{equation*}
	S_{1} = \int\limits_{\Gamma^J\setminus\mathbb{H}\times \mathbb{C}}  
	\phi(\tau,z) \overline{  \left(\mathcal{L}_{m}- \frac{m(2k-1)}{6}E_{2}\right)   \sum_{\gamma\in\Gamma_{\infty}^J\setminus \Gamma^J} {q^{\ell} \zeta^{w}}|_{k, m} \gamma  }~v^{k+2}e^{\dfrac{-4\pi my^2}{v}} dV_J.
	\end{equation*}
   Using \eqref{StokeHeatAction}, and changing the order of integration and summation,   we get
	\begin{equation*}
	S_{1} = \sum_{\gamma\in\Gamma_{\infty}^J\setminus \Gamma^J} 
	\int\limits_{\Gamma^J\setminus\mathbb{H}\times \mathbb{C}}  \phi(\tau,z) \overline{  \left(\mathcal{L}_{m}- \frac{m(2k-1)}{6}E_{2}\right) {q^{\ell} \zeta^{w}}|_{k+2, m} \gamma}~v^{k+2}e^{\dfrac{-4\pi my^2}{v}}  dV_J.
	\end{equation*}
	Changing the variable $(\tau, z)$ to $\gamma^{-1}.(\tau, z)$ gives rise to 
	\begin{equation*}
	S_{1}=  \sum_{\gamma\in\Gamma_{\infty}^J\setminus \Gamma^J} 
	\int\limits_{\gamma.\Gamma^J\setminus\mathbb{H}\times \mathbb{C}}  \phi(\tau,z) \overline{  \left(\mathcal{L}_{m}- \frac{m(2k-1)}{6}E_{2}\right) {q^{\ell} \zeta^{w}}}~v^{k+2}e^{\dfrac{-4\pi my^2}{v}}  dV_J. 
	\end{equation*}
	On applying the Rankin's unfolding argument, we get
	$$
	S_{1}= \int\limits_{\Gamma^J_{\infty}\setminus\mathbb{H}\times \mathbb{C}}  \phi(\tau,z) \overline{  \left(\mathcal{L}_{m}- \frac{m(2k-1)}{6}E_{2}\right) {q^{\ell} \zeta^{w}}}~v^{k+2}e^{\dfrac{-4\pi my^2}{v}} dV_J.
	$$
	Now, replacing $\phi(\tau,z)$ and $E_{2}(\tau)$ with their the Fourier expansions given in \eqref{FJExp} and \eqref{E2Exp}, respectively, and applying  the action of the operator $\mathcal{L}_{m}$ on ${q^{n} \zeta^{r}}$,  the last integral is given by 
	\begin{equation*}
	\begin{split}
	S_{1} & = \sum\limits_{\substack{n \ge 1, r \in\mathbb{Z}\\ 4nm-r^{2}>0}} c(n,r) \left((4nm-r^{2})- \frac{m(2k-1)}{6}\right)  \int\limits_{\Gamma^J_{\infty}\setminus\mathbb{H}\times \mathbb{C}} {q^{n} \zeta^{r}}  \overline{{q^{\ell} \zeta^{w}}}~v^{k+2}e^{\dfrac{-4\pi my^2}{v}} dV_J \\
	& \quad + ~ 4m(2k-1) \sum\limits_{\substack{n \ge 1, r \in\mathbb{Z}\\ 4nm-r^{2}>0}} c(n,r) \displaystyle{\sum_{p=1}^{\infty} \sigma_{1}(p)} \int\limits_{\Gamma^J_{\infty}\setminus\mathbb{H}\times \mathbb{C}} {q^{n} \zeta^{r}}  \overline{{q^{\ell + p} \zeta^{w}}}~v^{k+2}e^{\dfrac{-4\pi my^2}{v}} dV_J.
	\end{split}
	\end{equation*}
	Using \lemref{IntegralApprox} for the integrals, we have 
%	\begin{equation*}
%	\begin{split}
%	S_{1} & =  c(\ell,w) \left((4\ell m-w^{2})- \frac{m(2k-1)}{6}\right)   \frac{m^{k}}{ 2 {\pi}^{k+\frac{1}{2}}} ~ \frac{\Gamma\left(k+\frac{1}{2}\right)}{(4\ell m-w^{2})^{k+\frac{1}{2}}}  \\
%	& \quad + ~ 4m(2k-1)  \displaystyle{\sum_{p=1}^{\infty}   c(\ell+p,w) \sigma_{1}(p)}  \frac{m^{k}}{ 2 {\pi}^{k+\frac{1}{2}}} ~ \frac{\Gamma\left(k+\frac{1}{2}\right)}{(4(\ell+p) m-w^{2})^{k+\frac{1}{2}}}.
%	\end{split}
%	\end{equation*}
	\begin{equation}\label{SEvaluation}
	\begin{split}
	S_{1} & =   \frac{m^{k}{\Gamma\left(k+\frac{1}{2}\right)}}{ 2 {\pi}^{k+\frac{1}{2}}}  \bigg[  \frac{c(\ell,w)}{(4\ell m-w^{2})^{k+\frac{1}{2}}} \left((4\ell m-w^{2})- \frac{m(2k-1)}{6}\right)   \\
	& \hspace{3.0 cm}
	+ ~ 4m(2k-1)  \displaystyle{\sum_{p=1}^{\infty}    \sigma_{1}(p)} ~ \frac{ c(\ell+p,w)}{(4(\ell+p) m-w^{2})^{k+\frac{1}{2}}} \bigg].
	\end{split}
	\end{equation}
	From \eqref{AJacobiCoeff} and \eqref{SEvaluation}, we obtain the required result.
%	\begin{equation*}
%	\begin{split}
%	&c_{\mathcal{A}}(\ell,w)  = \dfrac {2 \pi^{k-\frac{3}{2}}}{m^{k-2}\Gamma(k-\frac{3}{2})} {(4m \ell-w^{2})^{k-\frac{3}{2}}} \times  S_{1} \\ 
%	&= \dfrac{m^{2} \left(k-\frac{1}{2}\right) \left(k-\frac{1}{2}\right)}{\pi^{2}{(4\ell m-w^{2})^{2}}}  \bigg[ 
%		  \frac{c(\ell,w)}{(4\ell m-w^{2})^{k+\frac{1}{2}}} \left((4\ell m-w^{2})- \frac{2m(2k-1)}{12}\right)  \\ 
%		  & \hspace{4.5 cm}
%		 +  4m(2k-1)  \displaystyle{\sum_{p=1}^{\infty}    \sigma_{1}(p)}  \frac{ c(\ell+p,w)}{(4(\ell+p) m-w^{2})^{k+\frac{1}{2}}}  \bigg].
%	\end{split}
%	\end{equation*}
This completes the proof.
	
\smallskip	
\noindent
\textbf{Proof of \thmref{ADOBD}:} %We now give a proof of \thmref{ADOBD}.  
Let $\phi(\tau, z) \in J^{cusp}_{k+2, m}$ with $(n,r)$-th  Fourier coefficient $c(n,r)$ given in  the Fourier expansion \eqref{FJExp}.  Let $\mathcal{O}^{*}$ denote the adjoint of $\mathcal{O}$.  Let the Fourier expansion of   $\mathcal{O}^{*}\phi$ is given by
$$
\mathcal{O}^{*}\phi(\tau, z)=  \sum_{\substack {\ell \ge 1, w \in \mathbb{Z}, \\ 4m \ell - w^{2} > 0}} c_{\mathcal{B}}(\ell, w) q^{\ell} \zeta^{w}. 
$$
	Now, we consider the $(\ell, w)$-th Poinca{\'r}e series of weight $k$ and index $m$ as given in \eqref{jacobi-poincare}. Using \lemref{jacobi-poincare-lemma}, we have 
$$
\pangle{\mathcal{O}^{*}\phi,  P_{k,m;(\ell,w)}}= \dfrac{m^{k-2}\Gamma(k-\frac{3}{2})}{2 \pi^{k-\frac{3}{2}}}  (4m \ell -w^2)^{-k+\frac{3}{2}} c_{\mathcal{B}}(\ell,w),
$$
%where 
%$$
%\alpha_{k,m} = \dfrac{m^{k-2}\Gamma(k-\frac{3}{2})}{2 \pi^{k-\frac{3}{2}}}.
%$$
%On the other hand, by definition of the adjoint map we have
%$$
%\pangle{\mathcal{O}^{*}\phi,  P_{k,m;(\ell,w)}}=  \pangle
%{ \phi,   \mathcal{O} P_{k,m;(\ell,w)}}= \pangle{ \phi,  \left(  D_{\tau}  - \frac{k}{12} E_{2}  - J_{1} D_{z}  + m J_{2} \right) P_{k,m;(\ell,w)}}. 
%$$
By definition of the adjoint map, we get 
\begin{equation}\label{BJacobiCoeff}
 c_{\mathcal{B}}(\ell,w) = \dfrac {2 \pi^{k-\frac{3}{2}} {(4m \ell-w^{2})^{k-\frac{3}{2}}}}{m^{k-2}\Gamma(k-\frac{3}{2})}\pangle{\phi,   \left( D_{\tau}  - \frac{k}{12} E_{2}  - J_{1} D_{z}  + m J_{2} \right) P_{k,m;(\ell,w)}}.
\end{equation}
Let $T$ denote the Petersson scalar product $\pangle{\phi,   \left( D_{\tau}  - \frac{k}{12} E_{2}  - J_{1} D_{z}  + m J_{2} \right) P_{k,m;(\ell,w)}}$. Then,
%By definition,  $\langle \phi,   \mathcal{O} P_{k,m;(\ell,w)}\rangle$ equals to $T$ (say)
\begin{equation*}
	T = \int\limits_{\Gamma^J\setminus\mathbb{H}\times \mathbb{C}}  
	\phi(\tau,z) \overline{  \left(  D_{\tau}  - \frac{k}{12} E_{2}  - J_{1} D_{z}  + m J_{2} \right)   \sum_{\gamma\in\Gamma_{\infty}^J\setminus \Gamma^J} {q^{\ell} \zeta^{w}}|_{k, m} \gamma }~v^{k+2}e^{\dfrac{-4\pi my^2}{v}} dV_J. 
\end{equation*}
	Now, using  similar to the arguments as done in the proof of   \thmref{ADLM}, we get % (using \eqref{StokeOBDAction}, and changing the order of integration and summation, and changing the variable $(\tau, z)$ to $\gamma^{-1}.(\tau, z)$ then applying Rankin unfolding argument), we get
	\begin{equation}\label{integral exp of T}
	T= \int\limits_{\Gamma^J_{\infty}\setminus\mathbb{H}\times \mathbb{C}}  \phi(\tau,z) \overline{ \left(  D_{\tau}  - \frac{k}{12} E_{2}  - J_{1} D_{z}  + m J_{2} \right)  {q^{\ell} \zeta^{w}}}~v^{k+2}e^{\dfrac{-4\pi my^2}{v}} dV_J.
	\end{equation}
	
	Now, replacing  $J_{1}(\tau,z)$, $J_{2}(\tau,z)$ and $E_{2}(\tau)$ with their Fourier expansions given in \eqref{TwistESeries} and  \eqref{E2Exp}, and applying  the action of $D_{\tau}$ and $D_{z}$ on the monomial ${q^{n} \zeta^{r}}$, we get  
\begin{equation}\label{ObedExp}
	\begin{split}
& \left(  D_{\tau}  - \frac{k}{12} E_{2}  - J_{1} D_{z}  + m J_{2} \right)  {q^{\ell} \zeta^{w}} \\
& \quad  \quad  = \left( \ell - \frac{k}{12} + \frac{w}{2} +\frac{m}{6} \right) {q^{\ell} \zeta^{w}} + 2k \sum_{p \ge 1} \sigma_{1}(p) {q^{\ell+p} \zeta^{w}} - w {q^{\ell} \frac{\zeta^{w+1}}{\z-1}} \\
& \quad  \qquad + w \sum_{p \ge 1} \left( \sum_{d \vert p}(\zeta^{d+w} - \zeta^{-d+w}) \right) {q^{\ell+p}} 
  -2m \sum_{p \ge 1} \left(\sum_{d \vert p}\frac{p}{d} (\zeta^{d+w} + \zeta^{-d+w}) \right) {q^{\ell+p}}.
\end{split}
\end{equation}
We substitute the  Fourier expansion of  $\phi(\tau,z)$ (given in \eqref{FJExp}) and series expansion \eqref{ObedExp} in \eqref{integral exp of T}, we have 
\begin{equation}\label{TExpression}
	\begin{split}
	T =  (I) + (II) + (III) + (IV) +(V) 
	\end{split}
	\end{equation}
where
\begin{equation*}
	\begin{split}
	  (I) &  = \left( \ell - \frac{k}{12} + \frac{w}{2} +\frac{m}{6} \right)  \sum\limits_{\substack{n \ge 1, r \in\mathbb{Z}\\ 4nm-r^{2}>0}} c(n,r) \int\limits_{\Gamma^J_{\infty}\setminus\mathbb{H}\times \mathbb{C}} {q^{n} \zeta^{r}}  \overline{{q^{\ell + p} \zeta^{w}}}~v^{k+2}e^{\dfrac{-4\pi my^2}{v}} dV_J. \\
	   (II)  &  =  2 k ~ \sum\limits_{\substack{n \ge 1, r \in\mathbb{Z}\\ 4nm-r^{2}>0}} c(n,r) \displaystyle{\sum_{p=1}^{\infty} \sigma_{1}(p)} \int\limits_{\Gamma^J_{\infty}\setminus\mathbb{H}\times \mathbb{C}} {q^{n} \zeta^{r}}  \overline{{q^{\ell + p} \zeta^{w}}}~v^{k+2}e^{\dfrac{-4\pi my^2}{v}} dV_J.  \\  
     (III)  &  = -w \sum\limits_{\substack{n \ge 1, r \in\mathbb{Z}\\ 4nm-r^{2}>0}} c(n,r)   \int \limits_{\Gamma_{\infty}^J\setminus\mathbb{H}\times \mathbb{C}}
q^{n} \zeta^{r} \frac{\overline{q^{\ell} \zeta^{w+1}}}{\overline{\z}-1}~ v^{k+2} e^{\frac{-4 \pi my^{2}}{v}}  dV_J. \\  
\end{split}
	\end{equation*}
\begin{equation*}
	\begin{split}
	    (IV)  &  =   w \sum_{p \ge 1}  \sum_{d \vert p}   \sum\limits_{\substack{n \ge 1, r \in\mathbb{Z}\\ 4nm-r^{2}>0}} c(n,r)  \!\!\!\! \int\limits_{\Gamma^J_{\infty}\setminus\mathbb{H}\times \mathbb{C}}  \!\!\!\! {q^{n} \zeta^{r}}  \overline{{q^{\ell + p} ( \zeta^{w+d}-  \zeta^{-w+d} )}}~v^{k+2}e^{\dfrac{-4\pi my^2}{v}} dV_J. \\ 
	    (V) & =   -2m   \sum_{p \ge 1} \sum_{d \vert p} \frac{p}{d}  \sum\limits_{\substack{n \ge 1, r \in\mathbb{Z}\\ 4nm-r^{2}>0}}   \!\!\!\! c(n,r)   \!\!\!\!  \!\!  \int\limits_{\Gamma^J_{\infty}\setminus\mathbb{H}\times \mathbb{C}}  \!\!\!\!   \!\!  {q^{n} \zeta^{r}}  \overline{{q^{\ell + p} ( \zeta^{w+d}+  \zeta^{-w+d} )}}~v^{k+2}e^{\dfrac{-4\pi my^2}{v}} dV_J.
	\end{split}
	\end{equation*}
Let $ \beta_{k,m}:= \frac{m^{k}{\Gamma\left(k+\frac{1}{2}\right)}}{ 2 {\pi}^{k+\frac{1}{2}}}$. Now, we   apply \lemref{IntegralApprox} to get 
\begin{equation*}
	\begin{split}
	(I) &=  \beta_{k,m} ~~A_1,\quad 
	(II) =2 k ~ \beta_{k,m}~~ A_2, \quad 
	(IV) =  w~  \beta_{k,m}~~A_4 \quad \text{and} \quad 
	(V) =  -2m~  \beta_{k,m}~~ ~~A_5.
\end{split}
\end{equation*}
%\begin{equation}\label{TExpression}
%\begin{split}
%T =  \frac{m^{k}{\Gamma\left(k+\frac{1}{2}\right)}}{ 2 {\pi}^{k+\frac{1}{2}}}   ~~ [ (I) + (II) + (III) + (IV) +(V)] 
%\end{split}
%\end{equation}
%where 
where $A_i$'s are given in \thmref{ADOBD} %{ADJSerre}
Now, we trace the integral appearing in $(III)$.   
\begin{equation}\label{eq:III}
	\begin{split}
	(III) 
	& =-w  \sum\limits_{\substack{n \ge 1, r \in\mathbb{Z}\\ 4nm-r^{2}>0}} c(n,r)  \int \limits_{\Gamma_{\infty}^J\setminus\mathbb{H}\times \mathbb{C}}
	q^{n} \zeta^{r} \frac{\overline{q^{\ell} \zeta^{w+1}}}{\overline{\z}-1}~ v^{k+2} e^{\frac{-4 \pi my^{2}}{v}}  dV_J \\
	& = -w  \sum\limits_{\substack{n \ge 1, r \in\mathbb{Z}\\ 4nm-r^{2}>0}} c(n,r)  \int_{0}^{1} \int_{0}^{\infty} \int_{0}^{1} \int_{-\infty}^{\infty}  
	q^{n} \zeta^{r} \frac{\overline{q^{\ell} \zeta^{w+1}}}{\overline{\z}-1}~ v^{k+2} e^{\frac{-4 \pi my^{2}}{v}}
	\frac{du dv dx dy}{v^{3}}.
	\end{split}
	\end{equation}
	We have that 
	\begin{equation}\label{ZetaExp}
		\frac{\z}{\z-1}=
		\begin{cases}
		&-\sum_{p\ge 0}
		\z^{p+1}\; \;\text{if}\;\; 0<\im(\z)<\infty,\\ \\
		&\sum_{p\ge 0}
		\z^{-p}\;\; \;\;\;\;\;\text{if} \;-\infty<\im(\z)<0.
		\end{cases}
	\end{equation}
	Using \eqref{eq:III} and \eqref{ZetaExp}, we get
	\begin{equation*}
	\begin{split}
     (III) 
      & =w  \sum\limits_{\substack{n \ge 1, r \in\mathbb{Z}, p\ge 0\\ 4nm-r^{2}>0}} c(n,r)  \int_{0}^{1} \int_{0}^{\infty} \int_{0}^{1} \int_{0}^{\infty} q^n\z^{r}\overline{q^{\ell} \zeta^{w+p+1}}
      v^{k+2}e^{\dfrac{-4\pi my^2}{v}} \frac{du dv dx dy}{v^{3}}\\
      &\quad -w  \sum\limits_{\substack{n \ge 1, r \in\mathbb{Z}, p\ge 0\\ 4nm-r^{2}>0}} c(n,r) \int_{0}^{1} \int_{0}^{\infty} \int_{0}^{1} \int_{-\infty}^{0}  q^n\z^{r}\overline{q^{\ell} \zeta^{w-p}}
      v^{k+2}e^{\dfrac{-4\pi my^2}{v}} \frac{du dv dx dy}{v^{3}}.\\
    %  &=w  \sum\limits_{\substack{n \ge 1, r \in\mathbb{Z}, p\ge 0\\ 4nm-r^{2}>0}} c(n,r)  \int_{0}^{1} \int_{0}^{\infty} \int_{0}^{1} \int_{0}^{\infty}  e^{-2\pi v(n+\ell)} e^{2\pi i(n- \ell)u} e^{-2\pi y(r+w+p+1)}e^{2\pi i(r-w-p-1)x}\\
    %  &\hspace{8cm} \times
    %  v^{k+2}e^{\dfrac{-4\pi my^2}{v}} \frac{du dv dx dy}{v^{3}}\\
   %   &\quad -w  \sum\limits_{\substack{n \ge 1, r \in\mathbb{Z}, p\ge 0\\ 4nm-r^{2}>0}} c(n,r)\int_{0}^{1} \int_{0}^{\infty} \int_{0}^{1} \int_{-\infty}^{0}     e^{-2\pi v(n+\ell)} e^{2\pi i(n- \ell)u} e^{-2\pi y(r+w-p)}e^{2\pi i(r-w+p)x}\\
    %  &\hspace{8cm} \times
    %  v^{k+2}e^{\dfrac{-4\pi my^2}{v}} \frac{du dv dx dy}{v^{3}}
	\end{split}
	\end{equation*}
Substituting $\tau = u+ iv$ and $z= x+iy$ and integrating over $u$ and $x$, we get
	\begin{align*}
		(III)=
		&w\sum_{p\ge 0}
		c(\ell, w+p+1)
		\int_{0}^{\infty} \int_{0}^{\infty}
		e^{-4\pi v\ell} 
	    e^{-4\pi y(w+p+1)}
	    v^{k+2}e^{\dfrac{-4\pi my^2}{v}} 
	    \frac{ dv  dy}{v^{3}}\\
	    &
	   - w\sum_{p\ge 0}
	    c(\ell, w-p)
	    \int_{0}^{\infty} \int_{-\infty}^{0}
	    e^{-4\pi v\ell} 
	    e^{-4\pi y(w-p)}
	    v^{k+2}e^{\dfrac{-4\pi my^2}{v}} 
	    \frac{ dv  dy}{v^{3}}
	\end{align*}
Substituting  $y = \sqrt{\frac{v}{m}}\big( X- \frac{\sqrt{v}}{2 \sqrt{m}} (w+p+1) \big)$ in the first part
	and $y=\sqrt{\frac{v}{m}}\big( X- \frac{\sqrt{v}}{2 \sqrt{m}} (w-p) \big)$ and then replace $X$ by $-X$ in the  second part of the RHS of  (III) above, we get
	\begin{align*}
		(III)=
		&\frac{w}{\sqrt{m}}\sum_{p\ge 0}
		c(\ell, w+p+1)
		\int_{0}^{\infty} 
		e^{-4\pi v \lp \ell -\frac{(w+p+1)^2}{4m} \rp} 
		v^{k-\frac{1}{2}}
		\lp
		\int_{\frac{1}{2}\sqrt{\frac{v}{m}}(w+p+1) }^{\infty}
		e^{-4\pi X^2}
		  dX \rp
		  dv
		  \\
		&	- \frac{w}{\sqrt{m}}
		\sum_{p\ge 0}
		c(\ell, p-w)
		\int_{0}^{\infty} 
		e^{-4\pi v \lp \ell -\frac{(p-w)^2}{4m} \rp} 
		v^{k-\frac{1}{2}}
		\lp
		\int_{\frac{1}{2}\sqrt{\frac{v}{m}}(p-w) }^{\infty} 
		e^{-4\pi X^2}
		dX
    	\rp
	    dv\\
	    &=\frac{w}{4\sqrt{m}}\sum_{p\ge 0}
	    c(\ell, w+p+1)
	    \int_{0}^{\infty} 
	    e^{-4\pi v \lp \ell -\frac{(w+p+1)^2}{4m} \rp} 
	    v^{k-\frac{1}{2}}
	    \lp
	    1-\err\bigg(  
	    \sqrt{\frac{\pi v}{m}}(w+p+1)\bigg) 
	    \rp
	    dv
	    \\
	    &	- \frac{w}{4\sqrt{m}}
	    \sum_{p\ge 0}
	    c(\ell, p-w)
	    \int_{0}^{\infty} 
	    e^{-4\pi v \lp \ell -\frac{(p-w)^2}{4m} \rp} 
	    v^{k-\frac{1}{2}}
	    \lp
	    1-\err\bigg(\sqrt{\frac{\pi v}{m}}(p-w)\bigg) 
	    \rp
	    dv\\
	    & = (III)_{1} + (III)_{2}~ (\text{say}).
	     \end{align*}
	     where $ \err(a)$ denotes the error function given by
	\[
	   \err(a)=
	   \frac{2}{\sqrt{\pi}}
	   \int_{0}^{a}
	   e^{-t^2} dt 
	\]
	with $\err(0)=1$ and $\err(\infty)=1.$
Now, we use the following integral: 
%\begin{equation}\label{jacobi-intv}
%\begin{split}
%\int_0^{\infty}e^{-4\pi n v} v^{k-3}\dfrac{\sqrt{v}~e^{\pi\dfrac{r^{2}v}{m}}}{2\sqrt{m}}dv  = \dfrac{1}{2\pi^{k-\frac{3}{2}}} \dfrac{m^{k-2}\Gamma(k-\frac{3}{2})}{(4nm-r^{2})^{k-\frac{3}{2}}}.
%\end{split}
%\end{equation}
%we use \eqref{jacobi-intv} and following Integral associated to error function $ \err(a)$ given by 
\begin{equation}\label{ErrInt}
\begin{split}
 \int_{0}^{\infty}  v^{k-\frac{1}{2}} ~ e^{-4\pi vM}  \err\bigg(\sqrt{\frac{\pi v}{m}}
 \bigg) dv 
 &   = \frac{2}{\sqrt{m}}\frac{\Gamma(k+1)}{{(4\pi M)}^{k+1}} ~~ {}_{2}F_{1}\left(\frac{1}{2}, k+1; \frac{3}{2}; -\frac{1}{4mM}\right)  \\
\end{split}
\end{equation}
for $k>-1$ and $m, M >0$ to get $(III)_{1}$ and  $(III)_{2}$ . The integral in \eqref{ErrInt} is obtained using Mathematica software. Thus, we get   
%\begin{equation}
%\begin{split}
% & \int_{0}^{\infty}  v^{k-\frac{1}{2}} ~ e^{-4\pi v \lp \ell -\frac{(p+w+1)^2}{4m} \rp}  \err\bigg(\sqrt{\frac{\pi v}{m}}(p+w+1)\bigg) dv \\ 
% & ~   = \frac{2(p+w+1)}{\pi^{k+1}} \frac{m^{+\frac{1}{2}} \Gamma(k+1)}{ (4 \ell ~m- (p+w+1)^{2})^{k+1}}~ {}_{2}F_{1}\left(\frac{1}{2}, k+1; \frac{3}{2}; \frac{-(p+w+1)^{2}}{4 \ell ~m - (p+w+1)^{2}}\right).  
%\end{split}
%\end{equation}
%\begin{equation}
%	\begin{split}
%	& \int_{0}^{\infty}  v^{k-\frac{1}{2}} ~ e^{-4\pi v \lp \ell -\frac{(p-w)^2}{4m} \rp}  \err\bigg(\sqrt{\frac{\pi v}{m}}(p-w)\bigg) dv   \\ 
%	& ~   = \frac{2(p-w)}{\pi^{k+1}} \frac{m^{k+\frac{1}{2}} \Gamma(k+1)}{ (4 \ell ~m- (p-w)^{2})^{k+1}} ~{}_{2}F_{1}\left(\frac{1}{2}, k+1; \frac{3}{2}; \frac{-(p-w)}{4 \ell ~m - (p-w)^{2}}\right) .
%	\end{split}
%\end{equation}
%This allows us to get 
\begin{align*}
		(III)_{1}  % & = \frac{w}{\sqrt{m}}\sum_{p\ge 0} c(\ell, w+p+1)
	                  %   \int_{0}^{\infty} e^{-4\pi v \lp \ell -\frac{(w+p+1)^2}{4m} \rp}  v^{k-\frac{1}{2}} \lp
		             %     \int_{\frac{1}{2}\sqrt{\frac{v}{m}}(w+p+1) }^{\infty} e^{-4\pi X^2} dX \rp dv \\
		           &  = \frac{w}{4\sqrt{m}}\sum_{p\ge 0} c(\ell, w+p+1)
	                       \int_{0}^{\infty}  e^{-4\pi v \lp \ell -\frac{(w+p+1)^2}{4m} \rp} v^{k-\frac{1}{2}}
	                        \lp 1-\err\bigg(  \sqrt{\frac{\pi v}{m}}(w+p+1)\bigg)  \rp dv  \\ 
	                %&= \frac{w}{4\sqrt{m}}\sum_{p\ge 0} c(\ell, w+p+1)
	                %      \bigg[ \frac{m^{k+\frac{1}{2} }\Ga(k+\frac{1}{2})}{(\pi(4m\ell-(w+p+1)^2))^{k+\frac{1}{2}}}\\
	              % &\hspace{5.5cm}-
	               %     \int_{0}^{\infty} \err\bigg(\sqrt{\frac{\pi v}{m}}(w+p+1)\bigg) e^{-4\pi v \lp \ell -\frac{(w+p+1)^2}          {4m} \rp}  v^{k-\frac{1}{2}} dv \bigg] \\
	                 %    &= \frac{w}{4\sqrt{m}}\sum_{p\ge 0} c(\ell, w+p+1)
	                  %    \bigg[ \frac{m^{k+\frac{1}{2} }\Ga(k+\frac{1}{2})}{(\pi(4m\ell-(w+p+1)^2))^{k+\frac{1}{2}}}\\
	              % &\hspace{.5cm}-
	                 %   \frac{2(p+w+1)}{\pi^{k+1}} \frac{m^{+\frac{1}{2}} \Gamma(k+1)}{ (4 \ell ~m- (p+w+1)^{2})^{k+1}}~ {}_{2}F_{1}\left(\frac{1}{2}, k+1; \frac{3}{2}; \frac{-(p+w+1)^{2}}{4 \ell ~m - (p+w+1)^{2}}\right) \bigg] \\ 
	                    & = \frac{w~ m^{k}\Ga(k+\frac{1}{2})}{4~ \pi^{k+\frac{1}{2}}} \sum_{p\ge 0} \frac{c(\ell, w+p+1)}{((4m\ell-(w+p+1)^2))^{k+\frac{1}{2}}} \\
	                    & \quad \times  \bigg[1- \frac{2 \Gamma(k+1)}{\sqrt{\pi}\Gamma(k+1/2)}  
	                    \frac{(w+p+1)}{((4m\ell-(w+p+1)^2))^{\frac{1}{2}}} {}_{2}F_{1}\left(\frac{1}{2}, k+1; \frac{3}{2}; \frac{-(p+w+1)^{2}}{4 \ell ~m - (p+w+1)^{2}}\right)  \bigg].                 
\end{align*}
Similarly, we have 
\begin{align*}
		(III)_{2}  %& =- \frac{w}{\sqrt{m}}\sum_{p\ge 0} c(\ell, p-w)
	                  %     \int_{0}^{\infty} e^{-4\pi v \lp \ell -\frac{(p-w)^2}{4m} \rp}  v^{k-\frac{1}{2}} \lp
		             %     \int_{\frac{1}{2}\sqrt{\frac{v}{m}}(p-w) }^{\infty} e^{-4\pi X^2} dX \rp dv \\
		           &  = -\frac{w}{4\sqrt{m}}\sum_{p\ge 0} c(\ell, p-w)
	                       \int_{0}^{\infty}  e^{-4\pi v \lp \ell -\frac{(p-w)^2}{4m} \rp} v^{k-\frac{1}{2}}
	                        \lp 1-\err\bigg(  \sqrt{\frac{\pi v}{m}}(p-w)\bigg)  \rp dv  \\ 
	               % &= -\frac{w}{4\sqrt{m}}\sum_{p\ge 0} c(\ell, p-w)
	                %     \bigg[ \frac{m^{k+\frac{1}{2} }\Ga(k+\frac{1}{2})}{(\pi(4m\ell-(p-w)^2))^{k+\frac{1}{2}}}\\
	              % &\hspace{5.5cm}-
	               %     \int_{0}^{\infty} \err\bigg(\sqrt{\frac{\pi v}{m}}(p-w)\bigg) e^{-4\pi v \lp \ell -\frac{(p-w)^2}          {4m} \rp}  v^{k-\frac{1}{2}} dv \bigg] \\
	               %      &= -\frac{w}{4\sqrt{m}}\sum_{p\ge 0} c(\ell, p-w)
	                %      \bigg[ \frac{m^{k+\frac{1}{2} }\Ga(k+\frac{1}{2})}{(\pi(4m\ell-(p-w)^2))^{k+\frac{1}{2}}}\\
	              % &\hspace{.5cm}-
	                %    \frac{2(p-w)}{\pi^{k+1}} \frac{m^{+\frac{1}{2}} \Gamma(k+1)}{ (4 \ell ~m- (p-w)^{2})^{k+1}}~ {}_{2}F_{1}\left(\frac{1}{2}, k+1; \frac{3}{2}; \frac{-(p-w)^{2}}{4 \ell ~m - (p-w)^{2}}\right) \bigg] \\ 
	                    & = -\frac{w~ m^{k}\Ga(k+\frac{1}{2})}{4~ \pi^{k+\frac{1}{2}}} \sum_{p\ge 0} \frac{c(\ell, p-w)}{((4m\ell-(p-w)^2))^{k+\frac{1}{2}}} \\
	                    & \quad \times  \bigg[1- \frac{2 \Gamma(k+1)}{\sqrt{\pi}\Gamma(k+1/2)}  
	                    \frac{(p-w)}{((4m\ell-(p-w)^2))^{\frac{1}{2}}} {}_{2}F_{1}\left(\frac{1}{2}, k+1; \frac{3}{2}; \frac{-(p-w)^{2}}{4 \ell ~m - (p-w)^{2}}\right)  \bigg].                 
\end{align*}
Now combining $(I)-(V)$ and using \eqref{BJacobiCoeff}, we get the $(\ell, w)$-th Fourier coefficient of $\mathcal{O}^{*}\phi$.
	\section{Applications}
In this section, we give some applications of our results. Let $T^{*}$ denote one of these operator $\mathcal{L}^{*}_{k,m}$,   $\mathcal{O}^{*} $ and ${\partial}^{J *}$. Suppose that  $\text{dim}(J_{k,m}^{cusp})= 0$, then for any $\phi \in J_{k+2,m}^{cusp}$, ~$T^{*} \phi \equiv 0$. This allows to get the relationship among the coefficients and special values of certain Dirichlet series.

For example, let us consider $k=8$, $m=1$ and $\phi = \phi_{10, 1}=\frac{1}{144}(E_6E_{4, 1}-E_4E_{6, 1})$. Let  $c_{10,1}(\ell,w)$ denote  $(\ell,w)$-coefficient of cusp form  $\phi_{10, 1}$. We have seen that    $\text{dim}(J_{8,1}^{cusp})= 0$. So, we get that $\mathcal{L}^{*}_{k,m} \phi_{10, 1} = 0$. This allows us to get the following relation among the coefficients:
\begin{equation*}
\begin{split}
   \quad c_{10,1}(\ell,w)& = \frac{60 (4 \ell-w)^{17/2}}{5/2-(4 \ell-w)}  \sum_{p=1}^{\infty}   ~ \frac{ \sigma_{1}(p) c_{10,1}(\ell+p,w) }{(4(\ell+p) -w^{2})^{\frac{17}{2}}}.\\
\end{split}
\end{equation*}
Similarly, we have $\mathcal{O}^{*} \phi_{10, 1} = 0$. This gives that 
 $$ A_{1} + A_{2} + A_{31} +A_{32} + A_{4} +A_{5} =0,   $$  
where $A_{i}$ are given in \thmref{ADOBD} with $c(\ell,w)$, replaced by $c_{10,1}(\ell,w)$.
% and taking $k=10$ and $m=1$, and 
%\begin{equation*}
%\begin{split}
%{\partial}^{J *} \phi_{10, 1} &= 0  \quad \Longrightarrow  \quad  A_{1} + A_{2} + A_{31} +A_{32} + A_{4} +A_{5}=0            \\
%\end{split}
%\end{equation*} 

Moreover, let  $\text{dim}(J_{k,m}^{cusp})= 1$, and let  it is generated by  $\psi(\tau,z)$.  Then for any $\phi \in J_{k+2,m}^{cusp}$, $T^{*} \phi = M_{\phi,\psi} \psi(\tau,z)$ for some constant $M_{\phi,\psi}$ . This allows to get the relationship between the coefficients of $\phi$ and $\psi$.

For example, let us consider $k=10$, $m=1$,   $\psi = \phi_{10, 1}=\frac{1}{144}(E_6E_{4, 1}-E_4E_{6, 1})$ and  $\phi=\phi_{12, 1}=\frac{1}{144}(E_4^2E_{4, 1}-E_6E_{6, 1})$.  Let $c_{10,1}(\ell,w)$ (resp. $c_{12,1}(\ell,w)$) denote the $(\ell,w)$-coefficient of $\phi_{10, 1}$ (resp. $\phi_{12, 1}$).  We know that  $\text{dim}(J_{10,1}^{cusp})= 1$ and it is generated by $\phi_{10, 1}$.
Hence, we have $\mathcal{L}^{*}_{k,m} \phi_{12, 1} = \alpha ~ \phi_{10, 1}  $ for some $\a \neq 0$. This allows us to get the following relation among the coefficients:
\begin{equation*}
\begin{split}
   \frac{(4 \ell-w)-19/2}{60 (4 \ell-w)^{21/2}} c_{12,1}(\ell,w)  + 76 \sum_{p=1}^{\infty}   ~ \frac{ \sigma_{1}(p) c_{12,1}(\ell+p,w) }{(4(\ell+p) -w^{2})^{\frac{17}{2}}}  = \frac{\alpha}{\beta_{10,1,\ell,w}} c_{10, 1}(\ell,w). 
 \end{split}
\end{equation*}
Similarly, we have $\mathcal{O}^{*} \phi_{12, 1} =  \beta  ~ \phi_{10, 1} $ for some $\b \neq 0$. This gives that
\begin{equation*}
\begin{split}
  A_{1} + A_{2} + A_{31} +A_{32} + A_{4} +A_{5}= \frac{\beta}{\beta_{10,1,\ell,w}}   c_{10, 1}(\ell,w)  
\end{split}
\end{equation*}
where $A_{i}$ are given in \thmref{ADOBD} with $c(\ell,w)$, replaced by $c_{12,1}(\ell,w)$.

% and taking $k=10$ and $m=1$, and 
%\begin{equation*}
%\begin{split}
%{\partial}^{J *} \phi_{12, 1} &=\frac{ \alpha-\beta}{4m-1}~ \phi_{10, 1}     \quad   \Longrightarrow   \quad     \\
%\end{split}
%\end{equation*} 

	%\newpage

\end{document}